\documentclass[11pt]{article}

\usepackage{amssymb,amsmath,amsthm}
\usepackage{mathrsfs}
\usepackage{mathtools}
\usepackage{color}
\usepackage{marginnote}
\usepackage{graphicx}

\numberwithin{equation}{section}
\newtheorem{thm}{Theorem}[section]

\newtheorem{prop}[thm]{Proposition}
\newtheorem{lem}[thm]{Lemma}
\newtheorem{cor}[thm]{Corollary}
\theoremstyle{definition}
\newtheorem{rem}[thm]{Remark}

\let\oldproofname=\proofname
\renewcommand{\proofname}{\rm\bf{\oldproofname}}

\newcommand{\N}{\mathbb{N}}
\newcommand{\Z}{\mathbb{Z}}

\newcommand{\R}{\mathbb{R}}
\newcommand{\C}{\mathbb{C}}

\newcommand{\cA}{\mathcal{A}}

\newcommand{\cC}{\mathcal{C}}

\newcommand{\cE}{\mathcal{E}}
\newcommand{\cF}{\mathcal{F}}
\newcommand{\cG}{\mathcal{G}}

\newcommand{\cI}{\mathcal{I}}

\newcommand{\cL}{\mathcal{L}}

\newcommand{\cO}{\mathcal{O}}

\renewcommand{\Re}{\mathop{\mathrm{Re}}}
\renewcommand{\Im}{\mathop{\mathrm{Im}}}
\newcommand{\dd}{\,{\rm d}}
\newcommand{\D}{{\rm d}}
\renewcommand{\div}{\mathop{\mathrm{div}}\nolimits}
\newcommand{\curl}{\mathop{\mathrm{curl}}}

\newcommand{\QED}{\mbox{}\hfill$\Box$}
\renewcommand{\:}{\thinspace :}

\newcommand{\rb}{\overline{r}}  
\newcommand{\gamb}{\overline{\gamma}}  
\newcommand{\Rp}{\R_+}
\newcommand{\Rpb}{\overline{\R}_+}

\begin{document}

\title{On the linear stability of vortex columns in the energy space}

\author{
{\bf Thierry Gallay}\\
Institut Fourier\\
Universit\'e Grenoble Alpes, CNRS\\
100 rue des Maths\\
38610 Gi\`eres, France\\
{\small\tt Thierry.Gallay@univ-grenoble-alpes.fr}
\and
{\bf Didier Smets}\\
Laboratoire Jacques-Louis Lions\\
Sorbonne Universit\'e\\
4, Place Jussieu\\
75005 Paris, France\\
{\small\tt Didier.Smets@sorbonne-universite.fr}}

\date{June 17, 2019}
\maketitle

\begin{abstract}
We investigate the linear stability of inviscid columnar vortices 
with respect to finite energy perturbations. For a large class 
of vortex profiles, we show that the linearized evolution group has 
a sub-exponential growth in time, which means that the associated 
growth bound is equal to zero. This implies in particular that the 
spectrum of the linearized operator is entirely contained in the 
imaginary axis. This contribution complements the results of 
our previous work \cite{GS1}, where spectral stability was established 
for the linearized operator in the enstrophy space. 
\end{abstract}

\section{Introduction}\label{sec1}
It is well known that radially symmetric vortices in two-dimensional
incompressible and inviscid fluids are stable if the vorticity 
distribution is a monotone function of the distance to the 
vortex center \cite{Arn,MP}. In a three-dimensional framework, 
this result exactly means that {\em columnar vortices} with no axial 
flow are stable with respect to two-dimensional perturbations, 
provided Arnold's monotonicity condition is satisfied. Vortex 
columns play an important role in nature, especially in atmospheric 
flows, and are also often observed in laboratory experiments \cite{AKO}. 
It is therefore of great interest to determine their stability 
with respect to arbitrary perturbations, with no particular 
symmetry, but this question appears to be very difficult and 
the only rigorous results available so far are sufficient conditions 
for {\em spectral stability}. 

In a celebrated paper \cite{Ke}, Lord Kelvin considered the particular
case of Rankine's vortex and proved that the linearized operator has a
countable family of eigenvalues on the imaginary axis. The
corresponding eigenfunctions, which are now referred to as {\em
Kelvin's vibration modes}, have been extensively studied in the
literature, also for more general vortex profiles \cite{FSJ,LL,RS}. An
important contribution was made by Lord Rayleigh in \cite{Ra}, who
gave a simple condition for spectral stability with respect to
axisymmetric perturbations. Rayleigh's criterion, which requires that
the angular velocity $\Omega$ and the vorticity $W$ have the same sign
everywhere, is actually implied by Arnold's monotonicity condition for
localized vortices.  In the non-axisymmetric case, the only stability
result one can obtain using the techniques introduced by Rayleigh is
restricted to perturbations in a particular subspace, where the
angular Fourier mode $m$ and the vertical wave number $k$ are
fixed. In that subspace, we have a sufficient condition for spectral 
stability, involving a quantity that can be interpreted as a 
local Richardson number. However, as is emphasized by Howard and 
Gupta \cite{HG}, that criterion always fails when the ratio $k^2/m^2$ 
is sufficiently small, and therefore does note provide any 
unconditional stability result.

In a recent work \cite{GS1}, we perform a rigorous mathematical study
of the linearized operator at a columnar vortex, using the vorticity
formulation of the Euler equations. We assume that the unperturbed
vorticity profile satisfies Arnold's monotonicity condition, hence
Rayleigh's criterion as well, and we impose an additional  
condition which happens to be satisfied in all classical examples and 
may only be technical. We work in the enstrophy space, assuming 
periodicity (with arbitrary period) in the vertical direction. In this 
framework, we prove that the spectrum of the linearized operator 
is entirely contained in the imaginary axis of the complex plane, 
which gives the first spectral stability result for columnar vortices 
with smooth velocity profile. More precisely, in any Fourier 
subspace characterized by its angular mode $m \neq 0$ and its 
vertical wave number $k \neq 0$, we show that the spectrum
of the linearized operator consists of an essential part that
fills an interval of the imaginary axis, and of a countable family 
of imaginary eigenvalues which accumulate only on
the essential spectrum (the latter correspond to Kelvin's vibration
modes). The most difficult part of our analysis is to preclude the
existence of isolated eigenvalues with nonzero real part, which can 
eventually be done by combining Howard and Gupta's criterion, a 
homotopy argument, and a detailed analysis of the eigenvalue equation 
when critical layers occur.

The goal of the present paper is to extend the results of \cite{GS1} in
several directions. First, we use the velocity formulation of the
Euler equations, and assume that the perturbations have finite
energy. This functional framework seems more natural than the
enstrophy space used in \cite{GS1}, but part of the analysis becomes
more complicated. In particular, due to the pressure term in the
velocity formulation, it is not obvious that the linearized operator
in a given Fourier sector is the sum of a (nearly) skew-symmetric
principal part and a compact perturbation. This decomposition, however, 
is the starting point of our approach, as it shows that the spectrum 
outside the imaginary axis is necessarily discrete. Also, unlike 
in \cite{GS1}, we do not have to assume periodicity in the vertical 
direction, so that our result applies to localized perturbations 
as well. Finally, we make a step towards linear stability by 
showing that the evolution group generated by the linearized 
operator has a mild, sub-exponential growth as $|t| \to \infty$. 
This is arguably the strongest way to express spectral stability. 

We now present our result in more detail. We consider the 
incompressible Euler equations in the whole space $\R^3$\: 
\begin{equation}\label{eq:Euler3d}
  \partial_t u + (u\cdot\nabla)u \,=\, -\nabla p\,, \qquad
  \div u \,=\, 0\,,
\end{equation}
where $u = u(x,t) \in \R^3$ denotes the velocity of the fluid at point
$x = (x_1,x_2,x_3) \in \R^3$ and time $t \in \R$, and
$p = p(x,t) \in \R$ is the associated pressure. The solutions we are
interested in are perturbations of flows with axial symmetry, and are
therefore conveniently described using cylindrical coordinates
$(r,\theta,z)$ defined by $x_1 = r\cos\theta$, $x_2 = r\sin\theta$,
and $x_3 = z$. The velocity field is decomposed as
\[ 
  u \,=\, u_r(r,\theta,z,t) e_r + u_\theta(r,\theta,z,t) e_\theta + 
  u_z(r,\theta,z,t) e_z\,,
\]
where $e_r$, $e_\theta$, $e_z$ are unit vectors in the radial, 
azimuthal, and vertical directions, respectively. The evolution equation 
in \eqref{eq:Euler3d} is then written in the equivalent form
\begin{equation}\label{eq:Eulercyl}
  \begin{split}
  \partial_t u_r + (u\cdot\nabla)u_r - \frac{u_\theta^2}{r} \,&=\, -\partial_r p\,, \\
  \partial_t u_\theta + (u\cdot\nabla)u_\theta + \frac{u_r u_\theta}{r} \,&=\, 
  -\frac1r \partial_\theta p\,, \\
  \partial_t u_z + (u\cdot\nabla)u_z \,&=\, -\partial_z p\,,
  \end{split}
\end{equation}
where $u\cdot \nabla = u_r \partial_r  + \frac1r u_\theta \partial_\theta  
+ u_z \partial_z$, and the incompressibility condition becomes 
\begin{equation}\label{eq:incomp}
  \div u \,=\, \frac1r\partial_r (ru_r) + \frac1r \partial_\theta u_\theta 
  + \partial_z u_z \,=\, 0\,. 
\end{equation}

Columnar vortices are described by stationary solutions of 
\eqref{eq:Eulercyl}, \eqref{eq:incomp} of the following form
\begin{equation}\label{eq:column}
  u \,=\, V(r) \,e_\theta\,,   \qquad p \,=\, P(r)\,,
\end{equation}
where the velocity profile $V :  \R_+ \to \R$ is arbitrary, and the 
pressure $P : \R_+ \to \R$ is determined by the centrifugal 
balance $rP'(r) = V(r)^2$. Other physically relevant 
quantities that characterize the vortex are the angular velocity
$\Omega$ and the vorticity $W$\:
\begin{equation}\label{eq:OmW}
  \Omega(r) \,=\, \frac{V(r)}{r}\,, \qquad 
  W(r) \,=\, \frac{1}{r}\,\frac{\D}{\D r}\bigl(r V(r)\bigr) \,=\, 
  r \Omega'(r) + 2 \Omega(r)\,.
\end{equation}

To investigate the stability of the vortex \eqref{eq:column}, 
we consider perturbed solutions of the form
 \[
  u(r,\theta,z,t) \,=\, V(r) \,e_\theta + \tilde u(r,\theta,z,t)\,, 
  \qquad p(r,\theta,z,t) \,=\, P(r) + \tilde p(r,\theta,z,t)\,.
\]
Inserting this Ansatz into \eqref{eq:Eulercyl} and neglecting the 
quadratic terms in $\tilde u$, we obtain the linearized evolution equations 
\begin{equation}\label{eq:upert}
  \begin{split}
  \partial_t u_r + \Omega \partial_\theta u_r - 2 \Omega u_\theta \,&=\, 
   -\partial_r p\,, \\
  \partial_t u_\theta + \Omega \partial_\theta u_\theta + W u_r \,&=\, 
  -\frac1r \partial_\theta p\,, \\
  \partial_t u_z + \Omega \partial_\theta  u_z \,&=\, -\partial_z p\,,
  \end{split}
\end{equation}
where we have dropped all tildes for notational simplicity. Remark that
the incompressibility condition \eqref{eq:incomp} still holds for the 
velocity perturbations. Thus, taking the divergence of both sides in 
\eqref{eq:upert}, we see that the pressure $p$ satisfies the second 
order elliptic equation
\begin{equation}\label{eq:pressure}
  -\partial_r^* \partial_r p -\frac{1}{r^2}\partial_\theta^2 p -\partial_z^2 p
  \,=\, 2 \bigl(\partial_r^* \Omega\bigr) \partial_\theta u_r - 2 \partial_r^* 
  \bigl(\Omega\,u_\theta\bigr)\,,
\end{equation}
where we introduced the shorthand notation $\partial_r^* f = \frac1r 
\partial_r(rf) = \partial_r f + \frac1r f$. 

We want to solve the evolution equation \eqref{eq:upert} in the Hilbert space
\[
  X \,=\, \Bigl\{u = (u_r,u_\theta,u_z) \in L^2(\R^3)^3\,\Big|\, 
  \partial_r^* u_r + \frac1r \partial_\theta u_\theta + \partial_z u_z 
  = 0\Bigr\}\,,
\]
equipped with the standard $L^2$ norm. Note that the definition of 
$X$ incorporates the incompressibility condition \eqref{eq:incomp}. 
In Section~\ref{sec3} we shall verify that, for any $u \in X$, the elliptic
equation \eqref{eq:pressure} has a unique solution (up to an 
irrelevant additive constant) that satisfies $\nabla p \in L^2(\R^3)^3$. 
Denoting that solution by $p = P[u]$, we can write Eq.~\eqref{eq:upert}
in the abstract form $\partial_t u = L u$, where $L$ is the integro-differential
operator in $X$ defined by
\begin{equation}\label{eq:Ldef}
  L u \,=\, \begin{pmatrix*}[l]
  -\Omega \partial_\theta u_r + 2 \Omega u_\theta -\partial_r P[u] \\[1mm]
  -\Omega \partial_\theta u_\theta - W u_r  -\frac1r\partial_\theta P[u] \\[1mm]
  -\Omega \partial_\theta u_z -\partial_z P[u]\end{pmatrix*}\,.
\end{equation}
If the angular velocity $\Omega$ and the vorticity $W$ are, for instance, bounded
and continuous functions on $\R_+$, it is not difficult to verify that the operator
$L$ generates a strongly continuous group of bounded linear operators in $X$, 
see Section~\ref{sec2}. Our goal is to show that, under additional assumptions
on the vortex profile, the norm of this evolution group has a mild growth as 
$|t| \to \infty$. Following \cite{GS1}, we make the following assumptions. 

\smallskip\noindent{\bf Assumption H1:} {\em The vorticity profile 
$W : \Rpb \to \Rp$ is a $\cC^2$ function satisfying $W'(0) = 0$, 
$W'(r) < 0$ for all $r > 0$, $r^3 W'(r) \to 0$ as $r \to \infty$, 
and}
\begin{equation}\label{eq:Winteg}
  \Gamma \,:=\, \int_0^\infty W(r) r\dd r \,<\, \infty\,.
\end{equation}

According to \eqref{eq:OmW}, the angular velocity $\Omega$ can be 
expressed in terms of the vorticity $W$ by the formula
\begin{equation}\label{eq:Omrep}  
  \Omega(r) \,=\, \frac{1}{r^2}\int_0^r W(s) s\dd s\,, \qquad r > 0\,,
\end{equation}
and the derivative $\Omega'$ satisfies
\begin{equation}\label{eq:Omrep2}  
  \Omega'(r) \,=\, \frac{W(r)-2\Omega(r)}{r} \,=\, \frac{1}{r^3}
  \int_0^r W'(s)s^2\dd s\,, \qquad r > 0\,.
\end{equation}
Thus $\Omega \in \cC^2(\Rpb) \cap C^3(\Rp)$ is a positive function 
satisfying $\Omega(0) = W(0)/2$, $\Omega'(0) = 0$, $\Omega'(r) < 0$ 
for all $r > 0$, and $r^2 \Omega(r) \to \Gamma$ as $r \to \infty$. 
Moreover, since $W$ is nonincreasing, it follows from \eqref{eq:Winteg} 
that $r^2 W(r) \to 0$ as $r \to \infty$, and this implies that $r^3 
\Omega'(r) \to -2\Gamma$ as $r \to \infty$. Similarly $r^4 \Omega''(r) \to
6\Gamma$ as $r \to \infty$. Finally, assumption H1 implies that the
{\em Rayleigh function} is positive\:
\begin{equation}\label{eq:Phidef}
  \Phi(r) \,=\, 2 \Omega(r) W(r) \,>\, 0\,, \qquad r \ge 0\,.
\end{equation}

\smallskip\noindent{\bf Assumption H2:} {\em The $\cC^1$ function 
$J : \Rp \to \Rp$ defined by
\begin{equation}\label{eq:Jdef}
  J(r) \,=\, \frac{\Phi(r)}{\Omega'(r)^2}\,, \qquad  r > 0\,,
\end{equation}
satisfies $J'(r) < 0$ for all $r > 0$ and $rJ'(r)\to 0$ as $r\to \infty$.} 

\smallskip
The reader is referred to the previous work \cite{GS1} for a discussion
of these hypotheses. We just recall here that assumptions H1, H2 
are both satisfied in all classical examples that can be found in the 
physical literature.  In particular, they hold for the Lamb-Oseen vortex\:
\begin{equation}\label{eq:LOvortex}
  \Omega(r) \,=\, \frac{1}{r^2}\Bigl(1 - e^{-r^2}\Bigr)\,,
  \qquad W(r) \,=\, 2\,e^{-r^2}\,, 
\end{equation}
and for the Kaufmann-Scully vortex\:
\begin{equation}\label{eq:KSvortex}
  \Omega(r) \,=\, \frac{1}{1+r^2}\,, \qquad
  W(r) \,=\, \frac{2}{(1+r^2)^2}\,.
\end{equation}

\smallskip
Our main result can now be stated as follows\:

\begin{thm}\label{thm:main}
Assume that the vorticity profile $W$ satisfies assumptions H1, H2
above.  Then the linear operator $L$ defined in \eqref{eq:Ldef} is
the generator of a strongly continuous group $(e^{tL})_{t \in \R}$
of bounded linear operators in $X$. Moreover, for any
$\epsilon > 0$, there exists a constant $C_\epsilon \ge 1$ such that
\begin{equation}\label{eq:eLbound}
  \|e^{tL}\|_{X \to X} \,\le\, C_\epsilon\,e^{\epsilon |t|}\,, \qquad
  \hbox{for all } t \in \R\,.
\end{equation}
\end{thm}

\begin{rem}\label{rem:growth}
Estimate \eqref{eq:eLbound} means that {\em growth bound} of the group 
$e^{tL}$ is equal to zero, see \cite[Section~I.5]{EN}.  
Equivalently, the spectrum of $e^{tL}$ is contained in the unit circle
$\{z \in \C\,|\, |z| = 1\}$ for all $t \in \R$. Invoking the 
Hille-Yosida theorem, we deduce from \eqref{eq:eLbound} that the spectrum 
of the generator $L$ is entirely contained in the imaginary axis of the 
complex plane, and that the following resolvent bound holds for any $a > 0$\:
\begin{equation}\label{eq:resbound}
  \sup\Bigl\{ \|(z - L)^{-1}\|_{X \to X} \,\Big|\, z \in \C\,,~|\Re(z)| \ge a\Bigr\}
   \,<\, \infty\,.
\end{equation}
In fact, since $X$ is a Hilbert space, the Gearhart-Pr\"uss theorem 
\cite[Section~V.1]{EN} asserts that the resolvent bound \eqref{eq:resbound}
is also equivalent to the group estimate \eqref{eq:eLbound}. 
\end{rem}

\begin{rem}\label{rem:Ceps}
The constant $C_\epsilon$ in \eqref{eq:eLbound} may of course blow up 
as $\epsilon \to 0$, but unfortunately our proof does not give any 
precise information. It is reasonable to expect that $C_\epsilon = 
\cO(\epsilon^{-N})$ for some $N > 0$, which would imply that 
$ \|e^{tL}\| = \cO(|t|^N)$ as $|t| \to \infty$, but proving such 
an estimate is an open problem. 
\end{rem}

\begin{rem}\label{rem:normal}
In \eqref{eq:LOvortex}, \eqref{eq:KSvortex}, and in all what follows, 
we always assume that the vortex profile is normalized so that 
$W(0) = 2$, hence $\Omega(0) = 1$. The general case can be 
easily deduced by a rescaling argument. 
\end{rem}

The rest of this paper is organized as follows. In Section~\ref{sec2},
we describe the main steps in the proof of Theorem~\ref{thm:main}.  In
particular, we show that the linearized operator \eqref{eq:Ldef} is
the generator of a strongly continuous group in the Hilbert space $X$,
and we reduce the linearized equations to a family of one-dimensional
problems using a Fourier series expansion in the angular variable
$\theta$ and a Fourier transform with respect to the vertical variable
$z$. For a fixed value of the angular Fourier mode $m \in \Z$ and of
the vertical wave number $k \in \R$, we show that the restricted
linearized operator $L_{m,k}$ is the sum of a (nearly) skew-symmetric
part $A_m$ and of a compact perturbation $B_{m,k}$. Actually, proving
compactness of $B_{m,k}$ requires delicate estimates on the pressure,
which are postponed to Section~\ref{sec3}. We then invoke the result
of \cite{GS1} to show that $L_{m,k}$ has no eigenvalue, hence no
spectrum, outside the imaginary axis. The last step in the proof
consists in showing that, for any $a \neq 0$, the resolvent norm
$\|(s - L_{m,k})^{-1}\|$ is uniformly bounded for all $m \in \Z$, all
$k \in \R$, and all $s \in \C$ with $\Re(s) = a$. This crucial bound
is obtained in Section~\ref{sec4} using a priori estimates for the
resolvent equation, which give explicit bounds in some regions of the
parameter space, combined with a contradiction argument which takes
care of the other regions. The proof of Theorem~\ref{thm:main} is thus
concluded at the end of Section~\ref{sec2}, taking for granted the
results of Sections~\ref{sec3} and \ref{sec4} which are the main
original contributions of this paper.

\bigskip\noindent{\bf Acknowledgements.}
This work was partially supported by grants ANR-18-CE40-0027
(Th.G.) and ANR-14-CE25-0009-01 (D.S.) from the ``Agence Nationale 
de la Recherche''. The authors warmly thank an anonymous referee
for suggesting a more natural way to prove compactness of the 
operator $B_{m,k}$, which is now implemented in Section~\ref{sec32}.

\section{Main steps of the proof}
\label{sec2}

The proof of Theorem~\ref{thm:main} can be divided into four main
steps, which are detailed in the following subsections. The 
first two steps are rather elementary, but the remaining two 
require more technical calculations which are postponed to 
Sections~\ref{sec3} and \ref{sec4}. 

\subsection{Splitting of the linearized operator}\label{sec21}

The linearized operator \eqref{eq:Ldef} can be decomposed as 
$L = A + B$, where $A$ is the first order differential operator
\begin{equation}\label{eq:Adef}
  A u \,=\, -\Omega(r)\partial_\theta u + r\Omega'(r)u_r  \,e_\theta\,,
\end{equation}
and $B$ is the nonlocal operator 
\begin{equation}\label{eq:Bdef}
  B u \,=\, -\nabla P[u] + 2\Omega u_\theta e_r - 2 (r\Omega)' u_r e_\theta\,.
\end{equation}
We recall that $W = r\Omega' + 2\Omega$, and that $P[u]$ denotes the solution 
$p$ of the elliptic equation \eqref{eq:pressure}. As is easily verified, 
both operators $A$ and $B$ preserve the incompressibility condition 
$\div u = 0$, and this is precisely the reason for which we 
included the additional term $r\Omega'(r)u_r  \,e_\theta$ in the 
definition \eqref{eq:Adef} of the advection operator $A$. 

\begin{lem}\label{lem:AB}
Under assumption H1, the linear operator $A$ is the generator of a strongly 
continuous group in the Hilbert space $X$, and $B$ is a bounded linear 
operator in $X$. 
\end{lem}

\begin{proof}
The evolution equation $\partial_t u = Au$ is equivalent to the system
\[
  \partial_t u_r + \Omega(r) \partial_\theta u_r \,=\, 0\,, \quad
  \partial_t u_\theta + \Omega(r) \partial_\theta u_\theta \,=\, r \Omega' u_r\,, \quad
 \partial_t u_z + \Omega(r) \partial_\theta u_z \,=\, 0\,,
\]
which has the explicit solution
\begin{align}\nonumber
  u_r(r,\theta,z,t) \,&=\, u_r\bigl(r,\theta-\Omega(r) t,z,0\bigr)\,, \\ 
  \label{eq:Agroup} 
  u_\theta(r,\theta,z,t) \,&=\, u_\theta\bigl(r,\theta-\Omega(r) t,z,0\bigr) 
  + r\Omega'(r)t \,u_r\bigl(r,\theta-\Omega(r) t,z,0\bigr)\,, \\ \nonumber
  u_z(r,\theta,z,t) \,&=\, u_z\bigl(r,\theta-\Omega(r) t,z,0\bigr)\,,
\end{align}
for any $t \in \R$. Under assumption H1, the functions $\Omega$ and 
$r \mapsto r \Omega'(r)$ are bounded on $\Rp$. With this 
information at hand, it is straightforward to verify that the formulas 
\eqref{eq:Agroup} define a strongly continuous group $(e^{tA})_{t \in \R}$ 
of bounded operators in $X$. Moreover, there exists a constant $C > 0$ 
such that $\|e^{tA}\|_{X \to X} \le C(1+|t|)$ for all $t \in \R$. 

On the other hand, in view of definition \eqref{eq:pressure}, the pressure 
$p = P[u]$ satisfies the energy estimate
\begin{equation}\label{eq:pressurest}
  \|\partial_r p\|_{L^2(\R^3)}^2 + \|\frac{1}{r} \partial_\theta
  p\|_{L^2(\R^3)}^2 + \|\partial_z p\|_{L^2(\R^3)}^2 \,\le\, 
  C\Bigl(\|u_r\|_{L^2(\R^3)}^2 + \|u_\theta\|_{L^2(\R^3)}^2\Bigr)\,,
 \end{equation}
which is established in Section~\ref{sec3}, see Remark~\ref{rem:p} below.  
This shows that $B$ is a bounded linear operator in $X$. 
\end{proof}

It follows from Lemma~\ref{lem:AB} and standard perturbation theory 
\cite[Section~III.1]{EN} that the linear operator $L = A+B$ is the 
generator of a strongly continuous group of bounded operators in $X$. 
Our goal is to show that, under appropriate assumptions on the vortex 
profile, this evolution group has a mild (i.e., sub-exponential) growth 
as $|t| \to \infty$, as specified in \eqref{eq:eLbound}. 

\subsection{Fourier decomposition}\label{sec22}

To fully exploit the symmetries of the linearized operator 
\eqref{eq:Ldef}, whose coefficients only depend on the radial 
variable $r$, it is convenient to look for velocities and
pressures of the following form
\begin{equation}\label{eq:upFour}
  u(r,\theta,z,t) \,=\, u_{m,k}(r,t)\,e^{im\theta}\,e^{ikz}\,,
  \qquad 
  p(r,\theta,z,t) \,=\, p_{m,k}(r,t)\,e^{im\theta}\,e^{ikz}\,,
\end{equation}
where $m \in \Z$ is the angular Fourier mode and $k \in \R$ is the
vertical wave number. Of course, we assume that $\overline{u_{m,k}} 
= u_{-m,-k}$ and $\overline{p_{m,k}} = p_{-m,-k}$ so as to obtain 
real-valued functions after summing over all possible values 
of $m,k$. When restricted to the Fourier sector
\begin{equation}\label{eq:Xmkdef}
  X_{m,k} \,=\, \Bigl\{u = (u_r,u_\theta,u_z) \in L^2(\Rp,r\dd r)^3\,\Big|\, 
  \partial_r^* u_r + \frac{im}{r} u_\theta + ik u_z = 0\Bigr\}\,,
\end{equation}
the linear operator \eqref{eq:Ldef} reduces to the one-dimensional
operator
\begin{equation}\label{eq:Lmkdef}
  L_{m,k} u \,=\, \begin{pmatrix*}[l]
  -im\Omega u_r + 2 \Omega u_\theta -\partial_r P_{m,k}[u] \\[1mm]
  -im\Omega u_\theta - W u_r  -\frac{im}{r}P_{m,k}[u] \\[1mm]
  -im\Omega u_z -ik P_{m,k}[u]\end{pmatrix*}\,,
\end{equation}
where $P_{m,k}[u]$ denotes the solution $p$ of the following elliptic 
equation on $\Rp$\:
\begin{equation}\label{eq:pmk}
  -\partial_r^* \partial_r p +\frac{m^2}{r^2}\,p + k^2 p
  \,=\, 2im \bigl(\partial_r^* \Omega\bigr) u_r - 2 \partial_r^* 
  \bigl(\Omega\,u_\theta\bigr)\,.
\end{equation}
As in Section~\ref{sec21}, we decompose $L_{m,k} = A_m + B_{m,k}$, where
\begin{equation}\label{eq:ABmkdef}
  A_m u \,=\, \begin{pmatrix*}[l]
  -im \Omega u_r \\[1mm] -im \Omega u_\theta + r\Omega'u_r \\[1mm]
  -im \Omega u_z  \end{pmatrix*}\,, \qquad 
  B_{m,k} u \,=\, \begin{pmatrix*}[l] 
   -\partial_r P_{m,k}[u] + 2 \Omega u_\theta \\[1mm]
   -\frac{im}{r} P_{m,k}[u] - 2(r\Omega)' u_r \\[1mm]
   -ik P_{m,k}[u]  \end{pmatrix*}\,.
\end{equation}

The following result is the analog of \cite[Proposition~2.1]{GS1} in the
present context. 

\begin{prop}\label{prop:AB} Assume that the vorticity 
profile $W$ satisfies assumption H1 and the normalization condition
$W(0) = 2$. For any $m \in \Z$ and any $k \in \R$,\\[1mm]
1) The linear operator $A_m$ defined by \eqref{eq:ABmkdef} is bounded 
in $X_{m,k}$ with spectrum given by
\begin{equation}\label{eq:spAm}
  \sigma(A_m) \,=\, \Bigl\{z \in \C\,\Big|\, z = -imb 
  \hbox{ for some } b \in [0,1]\Bigr\}\,;
\end{equation}
2) The linear operator $B_{m,k}$ defined by \eqref{eq:ABmkdef} is 
compact in $X_{m,k}$. 
\end{prop}

\begin{proof}
Definition \eqref{eq:ABmkdef} shows that $A_m$ is essentially the 
multiplication operator by the function $-im\Omega$, whose range is 
precisely the imaginary interval \eqref{eq:spAm} since the angular 
velocity is normalized so that $\Omega(0) = 1$. So the first assertion
in Proposition~\ref{prop:AB} is rather obvious, and can be established 
rigorously by studying the resolvent operator $(z-A_m)^{-1}$, see 
\cite[Proposition~2.1]{GS1}. The proof of the second assertion 
requires careful estimates on a number of quantities related  
to the pressure, and is postponed to Section~\ref{sec32}. 
\end{proof}

\subsection{Control of the discrete spectrum}\label{sec23}

For any  $m \in \Z$ and any $k \in \R$, it follows from 
Proposition~\ref{prop:AB} and Weyl's theorem \cite[Theorem~I.4.1]{EE}
that the {\em essential spectrum} of the operator $L_{m,k} = A_m + 
B_{m,k}$ is the purely imaginary interval \eqref{eq:spAm}, whereas 
the rest of the spectrum entirely consists of isolated eigenvalues 
with finite multiplicities\footnote{It is not difficult to verify that, 
in the present case, the various definitions of the essential
spectrum given e.g. in \cite[Section~I.4]{EE} are all equivalent.}. 
To prove spectral stability, it is therefore sufficient to show 
that $L_{m,k}$ has no {\em eigenvalue} outside the imaginary axis. 
Given any $s \in \C$ with $\Re(s) \neq 0$, the eigenvalue 
equation $(s - L_{m,k})u = 0$ is equivalent to the system
\begin{equation}\label{eq:eigsys}
  \begin{array}{l}
  \gamma(r) u_r - 2\Omega(r)u_\theta \,=\, -\partial_r p\,, \\[1mm]
  \gamma(r) u_\theta  + W(r)u_r \,=\, -\frac{im}{r} p\,, \\[1mm]
  \gamma(r) u_z  \,=\, -ik p\,, \end{array} \qquad\quad
  \partial_r^* u_r + \frac{im}{r} u_\theta + ik u_z \,=\, 0\,,
\end{equation}
where $\gamma(r) = s + im\Omega(r)$. If $(m,k) \neq (0,0)$, one can 
eliminate the pressure $p$ and the velocity components $u_\theta$, $u_z$
from system \eqref{eq:eigsys}, which then reduces to a scalar equation 
for the radial velocity only\:
\begin{equation}\label{eq:eigscalar}
   -\partial_r \biggl(\frac{r^2 \partial_r^* u_r}{m^2 + k^2 r^2}\biggr) 
   + \biggl\{1 + \frac{1}{\gamma(r)^2}\frac{k^2 r^2 
  \Phi(r)}{m^2 + k^2r^2} + \frac{imr}{\gamma(r)}\partial_r
  \Bigl(\frac{W(r)}{m^2+k^2r^2}\Bigr)\biggr\}u_r \,=\, 0\,,
\end{equation}
where $\Phi = 2\Omega W$ is the Rayleigh function. The derivation 
of \eqref{eq:eigscalar} is standard and can be found in many textbooks, 
see e.g. \cite[Section 15]{DR}. It is reproduced in Section~\ref{sec41} 
below in the more general context of the resolvent equation. 

The main result of our previous work on columnar vortices 
can be stated as follows. 

\begin{prop}\label{prop:GS1} {\bf\cite{GS1}} Under assumptions H1 and H2, 
the elliptic equation \eqref{eq:eigscalar} has no nontrivial solution  
$u_r \in L^2(\Rp,r\dd r)$ if $\Re(s) \neq 0$. 
\end{prop}

\begin{cor}\label{cor:GS1} Under assumptions H1 and H2, the 
operator $L_{m,k}$ in $X_{m,k}$ has no eigenvalue outside the 
imaginary axis. 
\end{cor}

\begin{proof}
Assume that $u \in X_{m,k}$ satisfies $L_{m,k} u = su$ for some complex 
number $s$ with $\Re(s) \neq 0$. If $m = k = 0$, the incompressibility condition 
shows that $\partial_r^* u_r = 0$, hence $u_r = 0$, and since 
$\gamma(r) = s \neq 0$ the second and third relations in \eqref{eq:eigsys} 
imply that $u_\theta = u_r=u_z = 0$. If $(m,k) \neq (0,0)$, the radial 
velocity $u_r$ satisfies \eqref{eq:eigscalar}, and 
Proposition~\ref{prop:GS1} asserts that $u_r = 0$. Using the 
relations \eqref{eq:uthetaexp}, \eqref{eq:uzexp} below (with 
$f = 0$), we conclude that $u_\theta = u_z = 0$. 
\end{proof}

\subsection{Uniform resolvent estimates}\label{sec24}

Under assumptions H1, H2, it follows from Proposition~\ref{prop:AB} and
Corollary~\ref{cor:GS1} that the spectrum of the linear operator
$L_{m,k} = A_m + B_{m,k}$ is entirely located on the imaginary axis.
Equivalently, for any $s \in \C$ with $\Re(s) \neq 0$, the resolvent
$(s - L_{m,k})^{-1}$ is well defined as a bounded linear operator in
$X_{m,k}$. The main technical result of the present paper, whose 
proof is postponed to Section~\ref{sec4} below, asserts that
the resolvent bound is uniform with respect to the Fourier parameters
$m$ and $k$, and to the spectral parameter $s \in \C$ if $\Re(s)$ is
fixed.

\begin{prop}\label{prop:main} 
Assume that the vortex profile satisfies assumptions H1, H2. 
Then for any real number $a \neq 0$, one has
\begin{equation}\label{eq:unifres}
  \sup_{\Re(s) = a} \,\sup_{m \in \Z} ~\sup_{k \in \R}~
  \bigl\|(s - L_{m,k})^{-1}\bigr\|_{X_{m,k} \to X_{m,k}} \,<\, \infty\,.
\end{equation}
\end{prop}

Equipped with the uniform resolvent estimate given by 
Proposition~\ref{prop:main}, it is now straightforward to 
conclude the proof of our main result. 

\medskip
\noindent{\bf End of the proof of Theorem~\ref{thm:main}.} 
We know from Lemma~\ref{lem:AB} that the operator $L$ defined 
by \eqref{eq:Ldef} is the generator of a strongly continuous 
group of bounded linear operators in the Hilbert space $X$. 
For any $a \neq 0$, we set
\begin{equation}\label{eq:Fdef}
  F(a) \,=\, \sup_{\Re(s) = a} \bigl\|(s - L)^{-1}\bigr\|_{X \to X} 
  ~\leq\, \sup_{\Re(s) = a} \,\sup_{m \in \Z} ~\sup_{k \in \R}~
  \bigl\|(s - L_{m,k})^{-1}\bigr\|_{X_{m,k} \to X_{m,k}}\,,
\end{equation}
where the last inequality follows from Parseval's theorem.  The
function $F : \R^* \to (0,\infty)$ defined by \eqref{eq:Fdef} is even
by symmetry, and a straightforward perturbation argument shows 
that
\[
  \frac{F(a)}{1 + |b| F(a)} \,\le\, F(a+b) \,\le\,
  \frac{F(a)}{1 - |b| F(a)}\,, 
\]
for all $a \neq 0$ and all $b \in \R$ with $|b| F(a) < 1/2$, 
so that $F$ is continuous.  Moreover, the Hille-Yosida theorem
\cite[Theorem~II.3.8]{EN} asserts that $F(a) = \cO(|a|^{-1})$ as
$|a| \to \infty$, and it follows that the resolvent bound
\eqref{eq:resbound} holds for any $a > 0$. In particular, given any
$\epsilon > 0$, the semigroup $\bigl(e^{t(L-\epsilon)}\bigl)_{t \ge 0}$ 
satisfies the assumptions of the Gearhart-Pr\"uss theorem
\cite[Theorem~V.1.11]{EN}, and is therefore uniformly bounded. This
gives the desired bound \eqref{eq:eLbound} for positive times, and a
similar argument yields the corresponding estimate for $t \le 0$. 
The proof of Theorem~\ref{thm:main} is thus complete. \QED

\section{Estimates for the pressure}
\label{sec3}

In this section, we give detailed estimates on the pressure
$p = P_{m,k}[u]$ satisfying \eqref{eq:pmk}. That quantity appears in
all components of the vector-valued operator $B_{m,k}$ introduced
in \eqref{eq:ABmkdef}, and our ultimate goal is to prove the last part
of Proposition~\ref{prop:AB}, which asserts that $B_{m,k}$ is a
compact operator in the space $X_{m,k}$.

We assume henceforth that the vorticity profile $W$ satisfies
assumption~H1 in Section~\ref{sec1}. To derive energy estimates, it is
convenient in a first step to suppose that the divergence-free vector
field $u \in X_{m,k}$ is smooth and has compact support in $(0,+\infty)$. 
As is shown in Proposition~\ref{prop:approximation} in the Appendix,
the family of all such vector fields is dense in $X_{m,k}$, and the 
estimates obtained in that particular case remain valid for all 
$u \in X_{m,k}$ by a simple continuity argument. With this observation 
in mind, we now proceed assuming that $u$ is smooth and 
compactly supported. 

Equation \eqref{eq:pmk} has a unique solution $p$ such that the
quantities $\partial_r p$, $mp/r$, and $kp$ all belong to
$L^2(\Rp,r\dd r)$; the only exception is the particular case
$m = k =0$ where uniqueness holds up to an additive constant. One
possibility to justify this claim is to return to the cartesian
coordinates and to consider the elliptic equation \eqref{eq:pressure}
for the pressure $p : \R^3 \to \R$, which can be written in the form
\begin{equation}\label{eq:pressurebis}
  - \Delta p \,=\, 2r\Omega'(r) \bigl(e_r, (\nabla u)e_\theta\bigr) - 
  2 \Omega(r) (\curl u)\cdot e_z\,,
\end{equation}
where $r = (x_1^2 + x_2^2)^{1/2}$. Note that the right-hand side is
regular (of class $\cC^2$ under assumption H1) and has compact support
in the horizontal variables. Eq.~\eqref{eq:pressurebis} thus holds in
the classical sense, and uniqueness up to a constant of a bounded 
solution $p$ is a consequence of Liouville's theorem for harmonic
functions in $\R^3$. In our framework, however, using
\eqref{eq:pressurebis} is not the easiest way to prove existence,
because according to \eqref{eq:upFour} we are interested in solutions
which depend on the variables $\theta$, $z$ in a specific way, and do
not decay to zero in the vertical direction.

If we restrict ourselves to the Fourier sector indexed by $(m,k)$,
existence of a solution to \eqref{eq:pmk} is conveniently established
using the explicit representation formulas collected in
Lemma~\ref{lem:rep} below. As can be seen from these expressions, the
solution $p$ of \eqref{eq:pmk} is smooth near the origin and satisfies
the homogeneous Dirichlet condition at the artificial boundary $r = 0$
if $|m| \ge 1$, and the homogeneous Neumann condition if $m = 0$ or
$|m| \ge 2$. As $r \to \infty$, it follows from \eqref{eq:repmk},
\eqref{eq:repm0} that $p(r)$ decays to zero exponentially fast if
$k \neq 0$, and behaves like $r^{-|m|}$ if $m \neq 0$ and $k = 0$.  In
the very particular case where $m = k = 0$, the pressure vanishes near
infinity if $u$ has compact support. Boundary conditions and decay
properties for the derivatives of $p$ can be derived in a similar way,
and will (often implicitly) be used in the proofs below to neglect boundary
terms when integrating by parts.

For functions or vector fields defined on $\Rp,$
we always use in the sequel the notation $\|\cdot\|_{L^2}$ to denote
the Lebesgue $L^2$ norm with respect to the measure $r\dd r$.
The corresponding Hermitian inner product will be denoted by 
$\langle \cdot,\cdot\rangle$.

\subsection{Energy estimates}\label{sec31}

Throughout this section, we assume that $u \in X_{m,k}$ and we denote by 
$p = P_{m,k}[u]$ the solution of \eqref{eq:pmk} given by 
Lemma~\ref{lem:rep}. We begin with a standard $L^2$ energy estimate.

\begin{lem}\label{lem:ell}
For any $u \in X_{m,k}$ we have  
\begin{equation}\label{eq:pmkest}
  \|\partial_r p\|_{L^2}^2 +  \Bigl\|\frac{mp}{r}\Bigr\|_{L^2}^2 + 
  \|kp\|_{L^2}^2 \,\le\, C \bigl(\|u_r\|_{L^2}^2 + 
  \|u_\theta\|_{L^2}^2\bigr)\,,
\end{equation}
where the constant $C > 0$ depends only on $\Omega$. 
\end{lem}

\begin{proof}
By density, it is sufficient to prove \eqref{eq:pmkest} under the 
additional assumption that $u$ is smooth and compactly supported 
in $(0,+\infty)$. To do so, we multiply both sides of \eqref{eq:pmk} 
by $r\bar p$ and integrate the result over $\Rp$. Integrating by parts 
and using H\"older's inequality, we obtain
\begin{align*}
  \|\partial_r p\|_{L^2}^2 +  \Bigl\|\frac{mp}{r}\Bigr\|_{L^2}^2 + 
  \|kp\|_{L^2}^2 \,&=\, \int_0^\infty \bar p\Bigl(\frac{2im}{r}
  (r\Omega)' u_r - 2 \partial_r^* (\Omega u_\theta)\Bigr)r\dd r \\
  \,&=\, \int_0^\infty \Bigl(\frac{2im}{r}(r\Omega)' \bar p u_r
  + 2(\partial_r \bar p) \Omega u_\theta\Bigr)r\dd r \\
  \,&\le\, 2 \|(r\Omega)'\|_{L^\infty} \Bigl\|\frac{mp}{r}\Bigr\|_{L^2}
  \|u_r\|_{L^2} + 2 \|\Omega\|_{L^\infty}\|\partial_r p\|_{L^2}
  \|u_\theta\|_{L^2}\,,
\end{align*}
hence
\[
  \|\partial_r p\|_{L^2}^2 +  \Bigl\|\frac{mp}{r}\Bigr\|_{L^2}^2 + 
  \|kp\|_{L^2}^2 \,\le\, 4 \|(r\Omega)'\|_{L^\infty}^2 \|u_r\|_{L^2}^2
  + 4 \|\Omega\|_{L^\infty}^2 \|u_\theta\|_{L^2}^2\,.
\]
Note that, by assumption H1, $\|(1{+}r)^2\Omega\|_{L^\infty} +
\|(1{+}r)^3\Omega'\|_{L^\infty} + \|(1{+}r)^4\Omega''\|_{L^\infty}
< \infty$. 
\end{proof}

\begin{rem}\label{rem:p}
The integrated pressure bound \eqref{eq:pressurest} can be established 
by an energy estimate as in the proof of Lemma~\ref{lem:ell}, 
or can be directly deduced from \eqref{eq:pmkest} using 
Parseval's theorem. 
\end{rem}

For later use, we also show that the solution $p = P_{m,k}[u]$ 
of \eqref{eq:pmk} depends continuously on the parameter $k$ 
as long as $k \neq 0$. 

\begin{lem}\label{lem:pcont}
Assume that $u_1 \in X_{m,k_1}$ and $u_2 \in X_{m,k_2}$, where
$m \in \Z$ and $k_1, k_2 \neq 0$. If we denote $p = P_{m,k_1}[u_1] 
- P_{m,k_2}[u_2]$, we have the estimate
\begin{equation}\label{eq:deltap}
  \|\partial_r p\|_{L^2}^2 +  \Bigl\|\frac{mp}{r}\Bigr\|_{L^2}^2 + 
  \|k_1 p\|_{L^2}^2 \,\le\, C \biggl(\|u_1 {-} u_2\|_{L^2}^2 + 
  \|u_2\|_{L^2}^2 \Bigl|\frac{k_1}{k_2} - \frac{k_2}{k_1}\Bigr|^2
  \biggr)\,,
\end{equation}
where the constant $C > 0$ depends only on $\Omega$. 
\end{lem}

\begin{proof}
In view of \eqref{eq:pmk}, the difference  $p = p_1 - p_2 \equiv 
P_{m,k_1}[u_1] - P_{m,k_2}[u_2]$ satisfies the equation
\[
  -\partial_r^* \partial_r p +\frac{m^2}{r^2}\,p + k_1^2 p
  \,=\, \frac{2im}{r}\bigl(r \Omega\bigr)'(u_{1,r} {-} u_{2,r}) 
  - 2 \partial_r^* \bigl(\Omega\,(u_{1,\theta} {-} u_{2,\theta})\bigr)
  + (k_2^2 - k_1^2)p_2\,.
\]
As in the proof of Lemma~\ref{lem:ell}, we multiply both sides
by $r\bar p$ and we integrate over $\Rp$. Integrating by 
parts and using H\"older's inequality, we easily obtain
\[
 \|\partial_r p\|_{L^2}^2 +  \Bigl\|\frac{mp}{r}\Bigr\|_{L^2}^2 + 
 \|k_1p\|_{L^2}^2 \,\le\, C \biggl(\|u_1 {-} u_2\|_{L^2}^2 + \|k_2p_2\|_{L^2}^2
 \Bigl|\frac{k_1}{k_2} - \frac{k_2}{k_1}\Bigr|^2\biggr)\,, 
\]
where the constant $C > 0$ depends only on $\|\Omega\|_{L^\infty}$ 
and $\|(r\Omega)'\|_{L^\infty}$. As $\|k_2 p_2\|_{L^2} \le C\|u_2\|_{L^2}$ 
by \eqref{eq:pmkest}, this gives the desired result. 
\end{proof}

Finally, we derive a weighted estimate which allows us to control 
the pressure $p = P_{m,k}[u]$ in the far-field region where $r \gg 1$. 

\begin{lem}\label{lem:pressionloin}
Assume that $k \neq 0$ or $|m| \ge 2$. If $u \in X_{m,k}$ and 
$p = P_{m,k}[u]$, then 
\begin{equation}\label{eq:pressionloin}
  \|r \partial_r p\|_{L^2}^2 +  \|m p\|_{L^2}^2 + \|k r p\|_{L^2}^2 \,\le\, 
  3 \|p\|_{L^2}^2 + C \bigl(\|u_r\|_{L^2}^2 + \|u_\theta\|_{L^2}^2\bigr)\,, 
\end{equation}
where the constant $C > 0$ depends only on $\Omega$. If $|m| \ge 2$
the first term in the right-hand side can be omitted. 
\end{lem}

\begin{proof}
We  multiply both sides of \eqref{eq:pmk} by $r^3\bar p$ and integrate 
the result over $\Rp$. Note that the integrand decays to zero 
exponentially fast if $k \neq 0$, and like $r^{1 - 2|m|}$ if $k \neq 0$, 
so that the integral converges if we assume that $k \neq 0$ or 
$|m| \ge 2$. After integrating by parts, we obtain the identity
\begin{equation}\label{eq:loin1}
   \|r \partial_r p\|_{L^2}^2 +  \|m p\|_{L^2}^2 + \|k r p\|_{L^2}^2 \,=\, 
   2 \|p\|_{L^2}^2 + \Re\bigl(I_1 + I_2\bigr)\,,
\end{equation}
where $I_1 = 2im \langle rp,(r\Omega)' u_r\rangle$ and $I_2 = 
2 \langle \partial_r (r^2 p),\Omega u_\theta\rangle =  2 \langle r 
\partial_r p,r\Omega u_\theta\rangle + 4 \langle p,r\Omega u_\theta\rangle$.
We observe that 
\begin{align*}
  |I_1| \,&\le\,  2 \|r(r\Omega)'\|_{L^\infty} \|mp\|_{L^2} \|u_r\|_{L^2} \,\le\, 
  \frac14 \|m p\|_{L^2}^2 + 4\|r(r\Omega)'\|_{L^\infty}^2 \|u_r\|_{L^2}^2\,, \\ 
  |I_2| \,&\le\, 2 \|r\Omega\|_{L^\infty} \Bigl(\|r\partial_r p\|_{L^2}  + 2 \|p\|_{L^2}
  \Bigr) \|u_\theta\|_{L^2} \,\le\, \frac14 \Bigl(\|r\partial_r p\|_{L^2}^2  
  + \|p\|_{L^2}^2 \Bigr) + 20 \|r\Omega\|_{L^\infty}^2 \|u_\theta\|_{L^2}^2\,,
\end{align*}
and replacing these estimates into \eqref{eq:loin1} we obtain 
\eqref{eq:pressionloin}. If $|m| \ge 2$, then $3\|p\|_{L^2}^2 \le \frac34
\|mp\|_{L^2}^2$, so that the first term in the right-hand side of
\eqref{eq:pressionloin} can be included in the left-hand side. 
\end{proof}

\begin{cor}\label{cor:comp1}
For any $m \in \Z$ and any $k \in \R$, the linear map $u \mapsto k P_{m,k}[u]$ 
from $X_{m,k}$ into $L^2(\Rp,r\dd r)$ is compact.  
\end{cor}

\begin{proof}
We can of course assume that $k \neq 0$. If $u$ lies in the unit ball 
of $X_{m,k}$, it follows from estimates \eqref{eq:pmkest} and 
\eqref{eq:pressionloin} that $\|k \partial_r p\|_{L^2}^2 + \|k r p\|_{L^2}^2 
\le C(k,\Omega)$ for some constant $C(k,\Omega)$ independent of $u$. Applying 
Lemma~\ref{lem:compcrit}, we conclude that the map $u \mapsto kp$ is compact. 
\end{proof}

\subsection{Compactness results}\label{sec32}

The aim of this section is to complete the proof of Proposition~\ref{prop:AB}, 
by showing the compactness of the linear operator $B_{m,k}$ defined in 
\eqref{eq:ABmkdef}. In view of Corollary~\ref{cor:comp1}, which already settles 
the case of the third component $B_{m,k,z}u := -ikP_{m,k}[u]$, we are left to 
prove that the linear mappings
\begin{align*}
  u \mapsto B_{m,k,r}u \,&:=\, -\partial_r P_{m,k}[u] + 2\Omega u_\theta\,, 
  \qquad \hbox{and} \\
  u \mapsto B_{m,k,\theta}u \,&:=\, -\frac{im}{r}P_{m,k}[u] - 2(r\Omega)' u_r\,,
\end{align*}
are compact from $X_{m,k}$ to $L^2(\Rp,r\dd r)$, for any $m \in \Z$ and any 
$k \in \R$. In the sequel, to simplify the notation, we write $B_r$, $B_\theta$, 
$B_z$ instead of $B_{m,k,r}u$, $B_{m,k,\theta}u$, $B_{m,k,z}u$, respectively. 

We first treat the simple particular case where $m = 0$. 

\begin{lem}\label{lem:estimBm0}
If $m = 0$ and $u \in X_{0,k}$, then
\begin{equation}\label{eq:estimBm0}
  \|\partial_r^* B_r\|_{L^2} + \|\partial_r^* B_\theta\|_{L^2} + \|rB_r\|_{L^2} 
  + \|rB_\theta\|_{L^2} \,\le\, C(k,\Omega) \|u\|_{L^2}\,,
\end{equation}
where the constant $C(k,\Omega)$ depends only on $k$ and $\Omega$. 
\end{lem}

\begin{proof}
If $m = 0$ and $k = 0$, then $\partial_r^* u_r = \div u = 0$, and this 
implies that $u_r = 0$. Similarly, the incompressibility condition for the 
vector $B$ implies that $B_r = 0$, and in view of \eqref{eq:ABmkdef}
it follows that $B$ vanishes identically. Thus estimate \eqref{eq:estimBm0} 
is trivially satisfied in that case. If $m = 0$ and $k \neq 0$, we deduce 
from \eqref{eq:pmkest} and \eqref{eq:pressionloin} that
\[
  \|rB_r\|_{L^2} \,\le\, \|r\partial_r p\|_{L^2} + 2\|r\Omega\|_{L^\infty} 
  \|u_\theta\|_{L^2} \,\le\, C(k,\Omega) \bigl( \|u_r\|_{L^2} + \|u_\theta\|_{L^2}
  \bigr)\,,
\]
and 
\[
  \|rB_\theta\|_{L^2} \,=\, 2\|r(r\Omega)'u_r\|_{L^2} \,\le\, 
  2 \|r(r\Omega)'\|_{L^\infty} \|u_r\|_{L^2}\,.
\]
As for the derivatives, we observe that $\partial_r^* B_r = -\partial_r^*\partial_r p 
+ 2\partial_r^*(\Omega u_\theta) = -k^2p$ in view of \eqref{eq:pmk}, and 
therefore we deduce from \eqref{eq:pmkest} that 
\[
  \|\partial_r^* B_r\|_{L^2} \,=\, \|k^2p\|_{L^2} \,\le\, C(k,\Omega) 
  \bigl(\|u_r\|_{L^2} + \|u_\theta\|_{L^2}\bigr)\,.
\] 
Finally, as $\partial_r^* u_r + ik u_z = 0$, we have $\partial_r^* B_\theta
= -2(r\Omega)''u_r - 2(r\Omega)'\partial_r^*u_r = -2(r\Omega)''u_r
+ 2ik(r\Omega)'u_z$, and it follows that
\[
  \|\partial_r^* B_\theta\|_{L^2} \,\le\, C(k,\Omega) \bigl(\|u_r\|_{L^2} + 
  \|u_z\|_{L^2}\bigr)\,.
\]
Collecting these estimates, we arrive at \eqref{eq:estimBm0}. 
\end{proof}

When $m \neq 0$, useful estimates on the vector $B_{m,k}u$ can be deduced from an
elliptic equation satisfied by the radial component $B_r$, which also involves
the quantities $R_1, R_2$ defined by
\begin{equation}\label{eq:R1R2def}
  R_1 \,=\, 2\bigl(-(r\Omega)'' u_r + im\Omega'u_\theta + ik(r\Omega)'u_z\bigr)\,,
 \quad\hbox{and}\quad 
  R_2 \,=\, \frac{2}{r}p + 2\Omega u_\theta\,.
\end{equation}
To derive that equation, we first observe that, in view of the
definitions \eqref{eq:ABmkdef} of $B_r, B_\theta$ and of the
incompressibility condition for $u$, the following relation holds
\begin{equation}\label{eq:elimBtheta}
  \partial_r\bigl(r B_\theta\bigr) \,=\, -im \partial_r p -2 \partial_r
  \bigl(r(r\Omega)'u_r\bigr) \,=\, imB_r + r R_1\,.
\end{equation}
Next, we have the incompressibility condition for $B$, which is equivalent to
\eqref{eq:pmk}\:
\begin{equation}\label{eq:divB}
  \partial_r^* B_r + \frac{im}{r}B_\theta + ikB_z \,=\, 0\,.
\end{equation}
If we multiply both members of \eqref{eq:divB} by $r^2$ and differentiate the
resulting identity with respect to $r$, we obtain using \eqref{eq:elimBtheta}
the desired equation
\begin{equation}\label{eq:Br}
  - \partial^2_r B_r - \frac{3}{r} \partial_r B_r + \Bigl(\frac{m^2-1}{r^2} 
  + k^2\Bigr) B_r \,=\,  \frac{im}{r}R_1 + k^2 R_2\,. 
\end{equation}

If $u \in X_{m,k}$ is smooth and compactly supported in $(0,+\infty)$,
it is clear from the definition \eqref{eq:ABmkdef} that the radial
component $B_r$ satisfies exactly the same boundary conditions at $r =
0$ as the pressure derivative $\partial_r p$, and has also the same
decay properties at infinity. In particular $B_r$ decays to zero
exponentially fast as $r \to \infty$ if $k \neq 0$, and behaves like
$r^{-1-|m|}$ if $k = 0$ and $m \neq 0$. These observations also apply
to the azimuthal component $B_\theta$.

We now exploit \eqref{eq:Br} to estimate $B_r$ and $B_\theta$, starting
with the general case where $|m| \ge 2$. 

\begin{lem}\label{lem:estimBm2}
If $|m|\ge 2$ and $u \in X_{m,k}$, then 
\begin{equation}\label{eq:estimBm2}
  \|\partial_r B_r\|_{L^2} + \|\partial_r^* B_\theta\|_{L^2} + \|rB_r\|_{L^2} 
  + \|rB_\theta\|_{L^2} \,\le\, C(m,k,\Omega) \|u\|_{L^2}\,,
\end{equation}
where the constant $C(m,k,\Omega)$ depends only on $m$, $k$ and $\Omega$. 
\end{lem}

\begin{proof}
We first observe that $\|R_1\|_{L^2} + \|r^2 R_1\|_{L^2} \le C \|u\|_{L^2}$, where the
constant depends on $m$, $k$, and $\Omega$. Similarly, in view of \eqref{eq:pmkest}
and \eqref{eq:pressionloin}, we have $\|krR_2\|_{L^2} + \|kr^2 R_2\|_{L^2} \le C
\|u\|_{L^2}$. Now, we multiply \eqref{eq:Br} by $r\bar B_r$ and integrate by parts.
This leads to the identity
\begin{equation}\label{eq:eb3}
  \|\partial_r B_r\|_{L^2}^2 + (m^2-1)\Bigl\|\frac{B_r}{r}\Bigr\|_{L^2}^2 + 
  \|kB_r\|_{L^2}^2 \,=\, \Re \langle B_r,\frac{im}{r}R_1 + k^2 R_2\rangle\,.
\end{equation}
To control the right-hand side, we use the estimates
\[
  \Bigl|\langle B_r,\frac{im}{r}R_1\rangle\Bigr| \,\le\, \Bigl\|\frac{mB_r}{r}
  \Bigr\|_{L^2} \|R_1\|_{L^2} \,\le\, C(m,k,\Omega)\Bigl\|\frac{mB_r}{r}\Bigr\|_{L^2} 
  \|u\|_{L^2}\,,
\]
and
\[
  \bigl| \langle B_r,k^2 R_2\rangle\bigr| \,\le\, 2\Bigl\|\frac{B_r}{r}\Bigr\|
  \|k^2 r R_2\|_{L^2} \,\le\, C(m,k,\Omega) \Bigl\|\frac{B_r}{r}\Bigr\|
  \|u\|_{L^2}\,.
\]
Inserting these bounds into \eqref{eq:eb3} and using Young's inequality 
together with the assumption that $|m| \ge 2$, we easily obtain
\begin{equation}\label{eq:eb4}
  \|\partial_r B_r\|_{L^2}^2 + m^2 \Bigl\|\frac{B_r}{r}\Bigr\|_{L^2}^2 + 
  k^2\|B_r\|_{L^2}^2 \,\le\, C(m,k,\Omega) \|u\|_{L^2}^2\,.
\end{equation}
In exactly the same way, if we multiply \eqref{eq:Br} by $r^3\bar B_r$
and integrate by parts, we arrive at the weighted estimate
\begin{equation}\label{eq:eb5}
  \|r\partial_r B_r\|_{L^2}^2 + m^2 \|B_r\|_{L^2}^2 + k^2 \|rB_r\|_{L^2}^2 
  \,\le\, C(m,k,\Omega) \|u\|_{L^2}^2\,.
\end{equation}

When $k = 0$, estimates \eqref{eq:eb4}, \eqref{eq:eb5} remain valid but they do 
not provide the desired control on $\|rB_r\|_{L^2}$. In that case, we multiply 
\eqref{eq:Br} by $r^5\bar B_r$ to derive the additional identity
\[
  \|r^2 \partial_r B_r\|_{L^2}^2 + (m^2-1)\|rB_r\|_{L^2}^2 \,=\, -2 \Re\langle 
  rB_r , r^2\partial_r B_r\rangle + \Re\langle r^4 B_r,\frac{im}{r} R_1\rangle\,.
\]
To estimate the right-hand side, we use the following bounds
\begin{align*}
  2\bigl|\langle r B_r, r^2\partial_r B_r\rangle\bigr| \,&\le\, \frac{3}{4}
  \|r^2\partial_r B_r\|_{L^2}^2 + \frac{4}{3} \|r B_r\|_{L^2}^2\,, \\
  \bigl|\langle rB_r,imr^2R_1\rangle\bigr| \,&\le\, \| mrB_r\|_{L^2} \|r^2R_1\|_{L^2} 
  \,\le\, C(m,\Omega) \|mrB_r\|_{L^2} \|u\|_{L^2}\,.
\end{align*}
Taking into account the assumption that $|m|\ge 2$, so that $\frac43 \le 
\frac{m^2-1}{2}$, we deduce that, for $k = 0$,
\begin{equation}\label{eq:eb6}
  \|r^2\partial_r B_r\|_{L^2}^2 + m^2 \|rB_r\|_{L^2}^2 \,\le\, C(m,\Omega) 
  \|u\|_{L^2}^2\,.
\end{equation}
Combining \eqref{eq:eb4}, \eqref{eq:eb5} and \eqref{eq:eb6} (when $k=0$), we obtain 
in particular the inequality 
\begin{equation}\label{eq:eb7}
  \|\partial_r B_r\|_{L^2} + \|rB_r\|_{L^2} \,\le\, C(m,k,\Omega) \|u\|_{L^2}\,.
\end{equation}

It remains to estimate the azimuthal component $B_\theta$, which satisfies  
$\partial_r^* B_\theta = \frac{im}{r}B_r + R_1$ by \eqref{eq:elimBtheta}.
Using inequalities \eqref{eq:eb4} and \eqref{eq:pressionloin} (in
the case where $|m| \ge 2$), we easily obtain
\begin{equation}\label{eq:eb8}
  \|\partial_r^* B_\theta\|_{L^2} + \|rB_\theta\|_{L^2} \,\le\, C(m,k,\Omega) 
  \|u\|_{L^2}\,,
\end{equation}
and estimate \eqref{eq:estimBm2} follows by combining \eqref{eq:eb7} and 
\eqref{eq:eb8}.
\end{proof}

The case where $m = \pm 1$ requires a slightly different argument, because 
an essential term in the elliptic equation \eqref{eq:Br} vanishes 
when $m^2 = 1$. It can be shown that this phenomenon is related to 
the translation invariance of the Euler equation in the original, 
cartesian coordinates. 

\begin{lem}\label{lem:estimBm1}
Assume that $m = \pm 1$, $k\neq 0$ and $u \in X_{m,k}$. Then
\begin{equation}\label{eq:estimBm1}
  \|\partial_r B_r\|_{L^2} + \|\partial_r^* D\|_{L^2} + \|rB_r\|_{L^2} + 
  \|rD\|_{L^2} \,\le\, C(k,\Omega) \|u\|_{L^2}\,,
\end{equation}
where $D = B_r + imB_\theta$ and the constant $C(k,\Omega)$ depends only 
on $k$ and $\Omega$. 
\end{lem}

\begin{proof}
We multiply both sides of \eqref{eq:Br} by $r^2\partial_r \bar B_r$,
take the real part, and integrate by parts. We obtain the identity
\[
  2\|\partial_r B_r\|_{L^2}^2 + k^2\|B_r\|_{L^2}^2 \,=\, -\Re
  \langle \partial_r B_r , im R_1 + k^2 r R_2\rangle\,,
\]
where the right-hand side is estimated as in the previous lemma.
This yields the bound
\begin{equation}\label{eq:eb9}
  \|\partial_r B_r\|_{L^2}^2 + k^2\|B_r\|_{L^2}^2 \,\le\, C(k,\Omega)\|u\|_{L^2}^2\,.
\end{equation}
In exactly the same way, multiplying \eqref{eq:Br} by $r^4\partial_r \bar B_r$,
we arrive at
\begin{equation}\label{eq:eb10}
  \|r \partial_r B_r\|_{L^2}^2 + k^2\|r B_r\|_{L^2}^2 \,\le\, C(k,\Omega)\|u\|_{L^2}^2\,.
\end{equation}
In particular, combining \eqref{eq:pmkest}, \eqref{eq:pressionloin}, \eqref{eq:eb9} 
and \eqref{eq:eb10}, we find
\begin{equation}\label{eq:eb11}
  \|D\|_{L^2} + \|rD\|_{L^2} \,\le\, C(k,\Omega)\|u\|_{L^2}\,.
\end{equation}
In addition, using the identity $\partial_r^* D = \partial_r B_r + \frac{1}{r}B_r 
+ im(\frac{im}{r}B_r + R_1) = \partial_r B_r + imR_1$, we obtain
\begin{equation}\label{eq:eb12}
  \|\partial_r^* D\|_{L^2} \,\le\, C(k,\Omega) \|u\|_{L^2}\,. 
\end{equation}
Estimate \eqref{eq:estimBm1} follows directly from \eqref{eq:eb9}--\eqref{eq:eb12}.
\end{proof}

\begin{lem}\label{lem:estimBm1k0}
Assume that $m = \pm 1$, $k = 0$, and $u \in X_{m,0}$. Then, for any $\alpha \in 
(0,1)$, 
\begin{equation}\label{eq:estimBm1k0}
  \|\partial_r B_r\|_{L^2} + \|\partial_r^* D\|_{L^2} + \|r^\alpha B_r\|_{L^2} 
  + \|r^\alpha D\|_{L^2} \,\le\, C(\alpha,\Omega) \|u\|_{L^2}\,,
\end{equation}
where $D = B_r + imB_\theta$ and the constant $C(\alpha, \Omega)$ depends only on 
$\alpha$ and $\Omega$. 
\end{lem}

\begin{proof}
If $|m| = 1$ and $k = 0$, equation \eqref{eq:Br} reduces to 
\[
  -\frac{1}{r^3}\partial_r\left(r^3 \partial_r B_r\right) \,=\, 
  \frac{im}{r}R_1\,,
\]
which can be explicitly integrated to give
\begin{equation}\label{eq:solexpl1}
  \partial_r B_r(r) \,=\, -\frac{im}{r^3} \int_0^r s^2R_1(s)\dd s\,, 
  \qquad r > 0\,,
\end{equation}
and finally
\begin{equation}\label{eq:solexpl2}
  B_r(r) \,=\, \frac{im}{2} \int_0^r R_1(s)\frac{s^2}{r^2} \dd s + 
  \frac{im}{2} \int_r^\infty R_1(s)\dd s\,, \qquad r > 0\,.
\end{equation}
Since $|m| = 1$ and $k = 0$, it follows from \eqref{eq:R1R2def} that
$\|r^{-1}R_1\|_{L^2} + \|r^3 R_1\|_{L^2} \le C \|u\|_{L^2}$, where the constant
depends only on  $\Omega$. Using that information, it is straightforward to
deduce from the representations \eqref{eq:solexpl1} and \eqref{eq:solexpl2}
that
\[
  \|\partial_r B_r\|_{L^2} + \|r^\alpha B_r\|_{L^2} 
  \,\le\, C(\alpha,\Omega) \|u\|_{L^2}\,,
\]
for any $\alpha < 1$. Note that $r B_r \notin L^2(\R_+,r\dd r)$ in general, 
because the first term in the right-hand side of \eqref{eq:solexpl2} 
decays exactly like $r^{-2}$ as $r \to \infty$. 

On the other hand, it follows from \eqref{eq:divB} that $imB_\theta = -r\partial_rB_r 
- B_r$, which implies that $D = -r\partial_r B_r$. Moreover, as in the previous 
lemma, we have $\partial_r^* D = \partial_r B_r + imR_1$. So, using the 
estimates above on $R_1$, we easily obtain the bound $\|\partial_r^* D\|_{L^2} 
+ \|r^\alpha D\|_{L^2} \le C\|u\|_{L^2}$, which concludes the proof. 
\end{proof}

\noindent
{\bf End of the Proof of Proposition \ref{prop:AB}.} When $|m|\neq 1$, in view 
of Lemmas~\ref{lem:estimBm0} and \ref{lem:estimBm2}, the compactness of the 
maps $u \mapsto B_{m,k,r}u$ and $u \mapsto B_{m,k,\theta}u$ is a direct
consequence of Lemma~\ref{lem:compcrit} in the Appendix. When $m = \pm1$, 
Lemma~\ref{lem:compcrit}, combined with Lemma~\ref{lem:estimBm1} or 
Lemma~\ref{lem:estimBm1k0}, shows that the maps $u \mapsto B_{m,k,r}u$ and 
$u \mapsto B_{m,k,r}u + imB_{m,k,\theta}u$ are compact, and so is 
the map $u \mapsto B_{m,k,\theta}u$. 
\qed

\begin{rem}\label{rem:oldversion}
It is also possible to obtain explicit representation formulas for 
the components of the vector-valued operator $B_{m,k}$ defined in 
\eqref{eq:ABmkdef}, and to use them to prove that the map 
$u \mapsto B_{m,k}u$ is compact in $X_{m,k}$. The computations, 
however, are rather cumbersome. That approach was followed in a 
previous version of this work \cite{GS2}. 
\end{rem}
 
\section{Resolvent bounds on vertical lines}
\label{sec4}

This final section is entirely devoted to the proof of 
Proposition~\ref{prop:main}. Let $a \neq 0$ be a nonzero real number. 
For any value of the angular Fourier mode $m \in \Z$, of the vertical wave 
number $k \in \R$, and of the spectral parameter $s \in \C$ with $\Re(s) = a$, 
we consider the resolvent equation $(s - L_{m,k})u = f$, which by definition
\eqref{eq:Lmkdef} is equivalent to the system
\begin{equation}\label{eq:eigsysf}
  \begin{array}{l}
  \gamma(r) u_r - 2\Omega(r)u_\theta \,=\, -\partial_r p + f_r\,, \\[1mm]
  \gamma(r) u_\theta  + W(r)u_r \,=\, -\frac{im}{r} p+f_\theta\,, \\[1mm]
  \gamma(r) u_z  \,=\, -ik p+f_z\,, \end{array} 
\end{equation}
where $\gamma(r) = s + im \Omega(r)$ and $p = P_{m,k}[u]$ is the 
solution of \eqref{eq:pmk} given by Lemma~\ref{lem:rep}. 
We recall that $u,f$ are divergence-free\: 
\begin{equation}\label{eq:incompuf}
 \partial_r^* u_r + \frac{im}{r} u_\theta + ik u_z \,=\, 0\,, \qquad
 \partial_r^* f_r + \frac{im}{r} f_\theta + ik f_z \,=\, 0\,.
\end{equation}
Our goal is to show that, given any $f \in X_{m,k}$, the (unique) 
solution $u \in X_{m,k}$ of \eqref{eq:eigsysf} satisfies 
$\|u\|_{L^2} \le C \|f\|_{L^2}$, where the constant $C > 0$ 
depends only on the spectral abscissa $a$ and on the angular 
velocity profile $\Omega$. In particular, the constant $C$ is 
independent of $m$, $k$, and $s$ provided $\Re(s) = a$. 

\begin{rem}\label{rem:symm}
It is interesting to observe how the resolvent system 
\eqref{eq:eigsysf}, \eqref{eq:incompuf} is transformed 
under the action of the following isometries\:
\[
\begin{array}{ll}
  \cI_1 \,:\, X_{m,k} \to X_{-m,k}\,, \quad & u \,\mapsto\, \tilde u \,:=\, 
  (u_r,-u_\theta,u_z)\,, \\
  \cI_2 \,:\, X_{m,k} \to X_{m,-k}\,, \quad & u \,\mapsto\, \hat u \,:=\, 
  (u_r,u_\theta,-u_z)\,, \\
  \cI_3 \,:\, X_{m,k} \to X_{-m,-k}\,, \quad & u \,\mapsto\, \bar u \,:=\, 
  (\bar u_r,\bar u_\theta,\bar u_z)\,,
 \end{array}
\]
where (as usual) $\bar u$ denotes the complex conjugate of $u$. 
If $u,f \in X_{m,k}$ and $s \in \C$, the resolvent equation 
$(s - L_{m,k})u = f$ is equivalent to any of the following 
three relations\:
\[
  \bigl(s + L_{-m,k}\bigr)\tilde u \,=\, \tilde f\,, \qquad
  \bigl(s - L_{m,-k}\bigr)\hat u \,=\, \hat f\,, \qquad
  \bigl(\bar s - L_{-m,-k}\bigr)\bar u \,=\, \bar f\,.
\]
This implies in particular that the spectrum of the operator $L_{m,k}$ 
in $X_{m,k}$ satisfies 
\begin{equation}\label{eq:symetriespectre}
  \sigma(L_{m,k}) \,=\, \sigma(L_{m,-k}) \,=\, -\sigma(L_{-m,k})\,,
  \qquad\hbox{and}\quad \sigma(L_{m,k}) \,=\, -\overline{\sigma(L_{m,k})}\,.
\end{equation}
As the spectrum $\sigma(L_{m,k})$ is symmetric with respect to the 
imaginary axis, due to the last relation in \eqref{eq:symetriespectre}, 
we can assume in what follows that the spectral abscissa $a$ is positive. 
Also, thanks to the first two relations, we can suppose without loss of 
generality that $m \in \N$ and $k \ge 0$. 
\end{rem}

\subsection{The scalar resolvent equation}\label{sec41}

A key ingredient in the proof of Proposition~\ref{prop:main} is the 
observation that the resolvent system \eqref{eq:eigsysf} is equivalent 
to a second order differential equation for the radial velocity $u_r$. 

\begin{lem}\label{lem:resODE}
Assume that $(k,m) \neq (0,0)$. If $u \in X_{m,k}$ is the solution of 
the resolvent equation \eqref{eq:eigsysf} for some $f \in X_{m,k}$, the 
radial velocity $u_r$ satisfies, for all $r > 0$, 
\begin{equation}\label{eq:eigscalarf}
  -\partial_r\bigl(\cA(r)\partial_r^* u_r\bigr) + \left(1 + \frac{k^2}{\gamma^2}
  \,\cA(r)\Phi(r) + \frac{imr}{\gamma}\partial_r\Bigl(\frac{W(r)}{m^2+k^2r^2}
  \Bigr)\right) u_r \,=\, \cF(r)\,,
\end{equation}
where  $\cA(r)\,=\,r^2/(m^2 + k^2 r^2)$ and 
\begin{equation}\label{eq:cFdef}
  \cF(r) \,=\, \frac{1}{\gamma}f_r + \cA\Bigl( \frac{2ik\Omega}{\gamma^2} 
   + \frac{2km}{\gamma}\frac{1}{m^2+k^2r^2}\Bigr)\bigl( -ikf_\theta + 
    \frac{im}{r}f_z\bigr) + \frac{im}{\gamma r}\cA\partial_r^* f_\theta 
    + \frac{ik}{\gamma} \cA \partial_r f_z\,.
\end{equation}
In addition, the azimuthal and vertical velocities are expressed in terms of 
$u_r$ by
\begin{align}\label{eq:uthetaexp}
  u_\theta \,&=\, \frac{im\cA}{r} \partial_r^*u_r -\frac{k^2\cA}{\gamma}\big(
  W u_r -f_\theta\big) - \frac{mk\cA}{\gamma r} f_z\,, \\ \label{eq:uzexp}
  u_z \,&=\, ik\cA \partial_r^*u_r +\frac{mk\cA}{\gamma r}
  \big( Wu_r-f_\theta\big) +\frac{m^2\cA}{\gamma r^2} f_z\,. 
\end{align}
\end{lem}

\begin{proof}
If we eliminate the  pressure $p$ from the last two lines in \eqref{eq:eigsysf}, 
we obtain
\begin{equation}\label{eq:ODE1}
  k W u_r + k\gamma u_\theta - \frac{\gamma m}{r} u_z \,=\, k f_\theta - 
  \frac{m}{r}f_z\,.
\end{equation}
This first relation can be combined with the incompressibility condition 
in \eqref{eq:incompuf} to eliminate the azimuthal velocity $u_\theta$. This gives
\begin{equation}\label{eq:ODE2}
  k\Bigl(\partial_r^* - \frac{imW}{\gamma r}\Bigr)u_r + i \Bigl(k^2 + \frac{m^2}{r^2}
  \Bigr) u_z \,=\, g_1 \,:=\, \frac{im^2}{\gamma r^2}f_z 
  - \frac{imk}{\gamma r}f_\theta\,, 
\end{equation}
which is \eqref{eq:uzexp}. As is easily verified, if in the previous
step we eliminate the vertical velocity $u_z$ from \eqref{eq:ODE1} and
\eqref{eq:incompuf}, we arrive at \eqref{eq:uthetaexp} instead of
\eqref{eq:uzexp}.

Alternatively, we can eliminate the pressure from the first and the
last line in \eqref{eq:eigsysf}.  This gives the second relation
\begin{equation}\label{eq:ODE3}
  ik \gamma u_r -2ik\Omega u_\theta - \partial_r (\gamma u_z) \,=\, ik f_r 
  -\partial_r f_z\,, 
\end{equation}
which can in turn be combined with \eqref{eq:ODE1} to eliminate the azimuthal 
velocity $u_\theta$. Using the relations $\gamma' = im\Omega'$ and 
$W = r\Omega' + 2\Omega$, we obtain in this way
\begin{equation}\label{eq:ODE4}
  \gamma^2 \Bigl(\partial_r + \frac{imW}{\gamma r}\Bigr)u_z -ik \bigl(\gamma^2 
  + \Phi\bigr)u_r \,=\, g_2 \,:=\, 2i\Omega \Bigl(\frac{m}{r}f_z - kf_\theta \Bigr) 
  + \gamma\Bigl(\partial_r f_z -ikf_r\Bigr)\,,
\end{equation}
where $\Phi = 2 \Omega W$ is the Rayleigh function. 

Now, we multiply the equality \eqref{eq:ODE2} by
$\cA = r^2/(m^2 + k^2 r^2)$ and apply the differential operator
$\partial_r + \frac{imW}{{\gamma r}}$ to both members of the resulting
expression. In view of \eqref{eq:ODE4}, we find
\begin{equation}\label{eq:ODE5}
   k \Bigl(\partial_r + \frac{imW}{\gamma r}\Bigr)\cA\Bigl(\partial_r^* - 
   \frac{imW}{\gamma r}\Bigr)u_r - k \Bigl(1 + \frac{\Phi}{\gamma^2}\Bigr)u_r
   \,=\, \Bigl(\partial_r + \frac{imW}{\gamma r}\Bigr)\cA g_1 - \frac{i}{\gamma^2}
   g_2\,.
\end{equation}
If $k \neq 0$, this equation is equivalent to \eqref{eq:eigscalarf}, as is easily 
verified by expanding the expressions in both sides of \eqref{eq:ODE5} and 
performing elementary simplifications. In the particular case where $k = 0$ (and
$m \neq 0$), equation \eqref{eq:eigscalarf} still holds but the derivation above
is not valid anymore. Instead, one must eliminate the pressure $p$ from 
the first two lines in \eqref{eq:eigsysf}, and then express the azimuthal 
velocity $u_\theta$ using the incompressibility condition. The details are 
left to the reader. 
\end{proof}

\begin{rem}\label{rem:obvious}
If $f = 0$, then $\cF = 0$ and Eq.~\eqref{eq:eigscalarf} reduces to the eigenvalue 
equation \eqref{eq:eigscalar}. 
\end{rem}

\begin{cor}\label{cor:uthetaz}
Under the assumptions of Lemma~\ref{lem:resODE}, we have the estimate
\begin{equation}\label{eq:uthetaz}
  \max\bigl\{\|u_\theta\|_{L^2}, \|u_z\|_{L^2}\bigr\} \,\le\, \|\cA^\frac12 \partial_r^* 
   u_r\|_{L^2} + \frac{1}{a}\Bigl(\|W\|_{L^\infty}\|u_r\|_{L^2} + \|f_\theta\|_{L^2}+ 
   \|f_z\|_{L^2}\Bigr)\,.
   \end{equation}
 \end{cor}
 
 \begin{proof}
 As $|\gamma(r)| \ge \Re(s) = a$ and  
\begin{equation}\label{eq:cAbounds}
   0 \,<\, \cA(r) \,\le\, \min\Bigl\{\frac{1}{k^2}, \frac{r^2}{m^2}\Bigr\}\,,
 \end{equation}
 estimate \eqref{eq:uthetaz} follows immediately from the representations 
 \eqref{eq:uthetaexp}, \eqref{eq:uzexp}. 
 \end{proof}

\subsection{Explicit resolvent estimates in particular cases}
\label{sec42}

We first establish the resolvent bound in the relatively simple case 
where $m = 0$, which corresponds to axisymmetric perturbations of the 
columnar vortex. 

\begin{lem}\label{lem:axi}
Assume that $m = 0$. For any $f \in X_{0,k}$, the solution 
$u \in X_{0,k}$ of \eqref{eq:eigsysf} satisfies
\begin{equation}\label{eq:axibound}
  \|u\|_{L^2} \,\le\, C_0\Bigl(\frac{1}{a} + \frac{1}{a^4}\Bigr)\|f\|_{L^2}\,,
\end{equation}
where the constant $C_0 > 0$ depends only on $\Omega$. 
\end{lem}

\begin{proof}
When $k = 0$, the incompressibility condition \eqref{eq:incompuf}
implies that $u_r = 0$, and since $\gamma(r) = s$ we deduce 
from the last two lines in \eqref{eq:eigsysf} that 
$u_\theta = f_\theta/s$ and $u_z = f_z/s$. As $|s| \ge \Re(s) = a$, 
we thus have $\|u\|_{L^2} \le \|f\|_{L^2}/a$, which is the desired 
conclusion. 

If $k \neq 0$, we assume without loss of generality that $k > 0$. 
Since $m = 0$, equation \eqref{eq:eigscalarf} satisfied by the radial velocity 
$u_r$ reduces to
\[
  -\partial_r \partial_r^* u_r + k^2\Bigl(1 + \frac{\Phi(r)}{s^2}\Bigr)u_r
  \,=\, \frac{k^2}{s}f_r + \frac{2k^2\Omega(r)}{s^2}f_\theta + 
  \frac{ik}{s}\partial_r f_z\,.
\]
We multiply both sides by $s r \bar u_r $ and integrate the resulting
equality over $\Rp$. After taking the real part, we obtain the 
identity
\[
  a\int_0^\infty \biggl\{|\partial_r^* u_r|^2 + k^2\Bigl(1 + 
  \frac{\Phi(r)}{|s|^2}\Bigr)|u_r|^2\biggr\}r\dd r \,=\, 
  \Re \int_0^\infty \bar u_r \Bigl(k^2f_r + \frac{2k^2\Omega(r)}{s}f_\theta + 
  ik\partial_r f_z\Bigr)r\dd r\,.
\]
As $\Phi(r) \ge 0$ by assumption H1, we easily deduce that
\[
  a \Bigl(\|\partial_r^* u_r\|_{L^2}^2 + k^2 \|u_r\|_{L^2}^2\Bigr)
  \,\le\, k^2 \|u_r\|_{L^2} \Bigl(\|f_r\|_{L^2} + \frac{2\|\Omega\|_{L^\infty}}{a}
  \|f_\theta\|_{L^2}\Bigr) + k \|\partial_r^* u_r\|_{L^2} \|f_z\|_{L^2}\,,
\]
and applying Young's inequality we obtain
\begin{equation}\label{eq:ur0}
  \frac{1}{k^2}\,\|\partial_r^* u_r\|_{L^2}^2 + \|u_r\|_{L^2}^2
  \,\le\, \frac{C}{a^2} \bigl(\|f_r\|_{L^2}^2 + \|f_z\|_{L^2}^2\bigr) 
  + \frac{C}{a^4}\|f_\theta\|_{L^2}^2\,, 
\end{equation}
where the constant $C > 0$ depends only on $\Omega$. 

With estimate \eqref{eq:ur0} at hand, we deduce from the second 
line in \eqref{eq:eigsysf} that
 \begin{equation}\label{eq:utheta0}
  \|u_\theta\|_{L^2} \,\le\, \frac{1}{|s|}\Bigl(\|W\|_{L^\infty} \|u_r\|_{L^2}
  + \|f_\theta\|_{L^2}\Bigr) \,\le\, C\Bigl(\frac{1}{a} 
  + \frac{1}{a^3}\Bigr)\|f\|_{L^2}\,.
\end{equation}
Similarly, using the third line in \eqref{eq:eigsysf} and 
estimate \eqref{eq:pmkest} for the pressure, we obtain
\begin{equation}\label{eq:uz0}
  \|u_z\|_{L^2} \,\le\, \frac{1}{|s|}\bigl(\|kp\|_{L^2} + \|f_z\|_{L^2}
  \bigr) \,\le\, \frac{C}{a}\bigl(\|u_r\|_{L^2} + \|u_\theta\|_{L^2}
 + \|f_z\|_{L^2}\bigr) \,\le\,  C\Bigl(\frac{1}{a} + \frac{1}{a^4}\Bigr)
 \|f\|_{L^2}\,.
\end{equation}
Combining \eqref{eq:ur0}, \eqref{eq:utheta0}, and \eqref{eq:uz0}, 
we arrive at \eqref{eq:axibound}.    
\end{proof}

In the rest of this section, we consider the more difficult case
where $m \neq 0$. In that situation, given any $s \in \C$ with 
$\Re(s) = a$, there exists a unique $b \in \R$ such that
\begin{equation}\label{eq:sdef}
  s \,=\, a - imb\,, \qquad \hbox{hence} \quad
  \gamma(r) \,=\, a + im\bigl(\Omega(r)-b\bigr)\,.
\end{equation}
Our goal is to obtain a resolvent bound that is uniform in the parameters
$m$, $k$, and $b$. In view of Remark~\ref{rem:symm}, we can assume 
without loss of generality that $m \ge 1$ and $k \ge 0$. 

Unlike in the axisymmetric case, we are not able to obtain here an 
explicit resolvent bound of the form \eqref{eq:axibound} for all 
values of the parameters $m$, $k$, and $b$. In some regions, 
we will have to invoke Proposition~\ref{prop:GS1}, which was 
established in \cite{GS1} using a contradiction argument that 
does not provide any explicit estimate of the resolvent operator. 
Nevertheless, our strategy is to obtain explicit bounds in the 
largest possible region of  the parameter space, and to rely 
on Proposition~\ref{prop:GS1} only when the direct approach 
does not work. 

We begin with the following elementary observation:

\begin{lem}\label{lem:grandgamma}
If $u,f\in X_{m,k}$ satisfy \eqref{eq:eigsysf}, then for any $M > 0$  
we have the estimate
\begin{equation}\label{eq:gg}
  \|1_{\{|\gamma|\ge M\}}\,u\|_{L^2} \,\le\, \frac{C_1}{M} \bigl( 
  \|u\|_{L^2} + \|f\|_{L^2}\bigr)\,,
\end{equation}
where the constant $C_1$ depends only on $\Omega$.
\end{lem}

\begin{proof}
We multiply all three equations in \eqref{eq:eigsysf} by $\gamma(r)^{-1} 
1_{\{|\gamma|\ge M\}}$ and take the $L^2$ norm of the resulting 
expression. Using estimate \eqref{eq:pmkest} to control the pressure, 
we arrive at \eqref{eq:gg}. 
\end{proof}

To obtain more general resolvent estimates, we exploit the differential
equation \eqref{eq:eigscalarf} satisfied by the radial velocity $u_r$. 
As a preliminary step, we prove the following result. 

\begin{lem}\label{lem:estimF}
If $u,f \in X_{m,k} \cap H^1(\Rp,r\dd r)$ and $\cF$ is defined by \eqref{eq:cFdef},
we have
\begin{equation}\label{eq:estimFur}
 \Bigl| \int_0^\infty \cF(r)\bar{u}_r r\dd r\Bigr| \,\le\, \frac{2}{a}\,
 \|\cA^\frac12 \partial_r^*u_r\|_{L^2}\|f\|_{L^2} + C_2 \Bigl(\frac{1}{a} 
  + \frac{1}{a^2}\Bigr)\|u_r\|_{L^2}\|f\|_{L^2}\,,
\end{equation}
where the constant $C_2$ depends only on $\Omega$. 
\end{lem}

\begin{proof}
We split the integral $\int_0^\infty \cF(r)\bar{u}_r r\dd r$ into four pieces, 
according to the expression of $\cF$ in \eqref{eq:cFdef}. As $|\gamma(r)| 
\ge \Re(s) = a$, the first term is easily estimated\:
\[
  \Bigl|\int_0^\infty \frac{1}{\gamma(r)} \,\bar{u}_r f_r r \dd r\Bigr| 
  \,\le\, \frac{1}{a}\,\|u_r\|_{L^2}\|f_r\|_{L^2}\,.
\]
As for the second term, we observe that $|k\cA(-ikf_\theta + \frac{im}{r}f_z)|
\le |f_\theta| + |f_z|$ by \eqref{eq:cAbounds}, so that
\[
 \Bigl| \int_0^\infty \cA\Bigl( \frac{2ik\Omega}{\gamma^2} + 
  \frac{2km}{\gamma}\frac{1}{m^2+k^2r^2}\Bigr)\bigl(-ikf_\theta + 
  \frac{im}{r}f_z\bigr)\bar{u}_r r\dd r\Bigr| \,\le\, \Bigl(\frac{2}{a^2}
  +\frac{2}{am}\Bigr)\|u_r\|_{L^2}\bigl(\|f_\theta\|_{L^2} +\|f_z\|_{L^2}\bigr)\,.
\]
The third term is integrated by parts as follows\:
\[
 \int_0^\infty \bar{u}_r im\frac{\cA}{\gamma r^2}\partial_r(rf_\theta)\,r\dd r 
 \,=\, -\int_0^\infty (\partial_r^*\bar{u}_r)im\frac{\cA}{\gamma r} f_\theta\,r\dd r - 
 \int_0^\infty im\bar{u}_r \partial_r\Bigl(\frac{\cA}{\gamma r^2}\Bigr)rf_\theta
 \,r\dd r\,. 
\]
Since $|m\cA^\frac12/r|\le 1$ by \eqref{eq:cAbounds}, we have on the one hand 
\[
 \Bigl|\int_0^\infty (\partial_r^*\bar{u}_r)im\frac{\cA}{\gamma r}\,f_\theta\,r
  \dd r\Bigr| \,\le\,  \frac{1}{a}\,\|\cA^\frac12 \partial_r^*u_r\|_{L^2}
  \|f_\theta\|_{L^2}\,, 
\]
and on the other hand 
\[
 \Bigl|mr\partial_r\Bigl(\frac{\cA}{\gamma r^2}\Bigr)\Bigr| \,=\, 
 \Bigl|\frac{im^2\Omega'\cA}{\gamma^2 r} + \frac{2mk^2\cA^2}{\gamma r^2}
 \Bigr| \,\le\, \frac{\|r\Omega'\|_{L^\infty}}{a^2} + \frac{2}{am}\,,
\]
so that 
\[
  \Bigl|\int_0^\infty im\partial_r\Bigl(\frac{\cA}{\gamma r^2}\Bigr)r f_\theta 
  \bar{u}_r\,r\dd r\Bigr| \,\le\, C\Bigl(\frac{1}{a^2} + \frac{1}{am}\Bigr)
  \|u_r\|_{L^2} \|f_\theta\|_{L^2}\,.
\]
In a similar way, the fourth and last term is integrated by parts\:
\begin{equation}\label{eq:pouru}
 \int_0^\infty \bar{u}_r \frac{ik}{\gamma}\,\cA\partial_rf_z\,r\dd r \,=\,
 -\int_0^\infty (\partial_r^*\bar{u}_r)\frac{ik}{\gamma}\,\cA f_z\,r\dd r - 
 \int_0^\infty ik\bar{u}_r \partial_r\Bigl(\frac{\cA}{\gamma}\Bigr) f_z\,r\dd r\,.
\end{equation}
Since $|k\cA^\frac12|\le 1$ by \eqref{eq:cAbounds}, we have
\[
 \Bigl|\int_0^\infty(\partial_r^*\bar{u}_r)\frac{ik}{\gamma}\cA f_z\,r\dd r\Bigr| 
 \,\le\, \frac{1}{a}\,\|\cA^\frac12 \partial_r^*u_r\|_{L^2}\|f_z\|_{L^2}\,.
\]
Moreover, using the relations $r\cA' = 2\cA(1-k^2\cA)$ and $\gamma' = im\Omega'$, 
we can estimate the last integral in \eqref{eq:pouru} as follows\:
\begin{align*}
  \Bigl|\int_0^\infty ik\bar{u}_r\partial_r\Bigl(\frac{\cA}{\gamma}\Bigr)f_z
  \,r\dd r\Bigr| \,&\le\,  \Big\|\frac{2k\cA}{\gamma r}(1-k^2\cA)-imk\frac{
  \Omega'\cA}{\gamma^2}\Big\|_{L^\infty} \|u_r\|_{L^2}\|f_z\|_{L^2}\\
  \,&\le\, \Bigl(\frac{2}{am} + \frac{\|r\Omega'\|_{L^\infty}}{a^2}\Bigr)
  \|u_r\|_{L^2}\|f_z\|_{L^2}\,.
\end{align*}
Collecting all estimates above and recalling that $m \ge 1$, we arrive 
at \eqref{eq:estimFur}. 
\end{proof}

We next establish an explicit estimate that will be useful when
the vertical wave number $k$ is small compared to the angular 
Fourier mode $m$.  

\begin{lem}\label{lem:kpetit}
If $m \ge 1$ and $u,f \in X_{m,k}$ satisfy \eqref{eq:eigsysf}, we have 
the estimate
\begin{equation}\label{eq:kpetit}
  \|\cA^\frac12 \partial_r^*u_r\|_{L^2}^2 +\|u_r\|_{L^2}^2 \,\le\, 
  C_3\Bigl(\frac{1}{a^2} + \frac{1}{a^4}\Bigr)\,\frac{k^2}{m^2+k^2}\,\|u_r\|_{L^2}^2 
  + C_3 \Bigl(\frac{1}{a^2} + \frac{1}{a^6}\Bigr)\|f\|_{L^2}^2\,,
\end{equation}
where the constant $C_3 > 0$ depends only on $\Omega$.  
\end{lem}

\begin{proof}
We start from the scalar resolvent equation \eqref{eq:eigscalarf} 
satisfied by the radial velocity $u_r$. Multiplying both sides 
by $r\bar{u}_r$ and integrating the resulting expression over $\Rp$, 
we obtain the following identity\:
\begin{equation}\label{eq:pouru0}
 \|\cA^\frac12 \partial_r^*u_r\|_{L^2}^2 +\|u_r\|_{L^2}^2 + I_1 + I_2 
 \,=\, \int_0^\infty \cF(r)\bar{u}_r r\dd r\,,
\end{equation}
where $\cF(r)$ is defined in \eqref{eq:cFdef} and 
\begin{equation}\label{eq:I12def}
 I_1 \,=\, \int_0^\infty \frac{k^2}{\gamma^2}\,\cA\Phi\,|u_r|^2r\dd r\,, 
 \qquad 
 I_2 \,=\, \int_0^\infty \frac{imr}{\gamma}\partial_r\Bigl(\frac{W}{m^2{+}k^2r^2}\Bigr)
 |u_r|^2r\dd r\,.
\end{equation}
The right-hand side of \eqref{eq:pouru0} is estimated in Lemma~\ref{lem:estimF}. 
On the other hand, using \eqref{eq:cAbounds} and the fact that  $|\gamma(r)| 
\ge \Re(s) = a$, we can bound
\begin{equation}\label{eq:pouru1}
  \Big|\frac{k^2}{\gamma^2}\cA\Phi\Big| \,\le\, 
  \min\biggl\{\frac{\|\Phi\|_{L^\infty}}{a^2}\,,\,\frac{k^2}{a^2m^2}\,
  \|r^2\Phi\|_{L^\infty}\biggr\}\,, \quad \hbox{so that}\quad  |I_1| 
  \,\le\,  \frac{C}{a^2}\,\frac{k^2}{m^2+k^2}\, \|u_r\|_{L^2}^2\,.
\end{equation}
Moreover, we have
\begin{equation}\label{eq:pouru2}
  \Big| \frac{mr}{\gamma}\partial_r\Bigl(\frac{W}{m^2+k^2r^2}\Bigr)\Big| 
  \,\le\, \frac{1}{am}\bigl(\|rW'\|_{L^\infty} + 2 \|W\|_{L^\infty}\bigl)\,, 
  \quad \hbox{so that}\quad  |I_2| \,\le\, \frac{C}{am}\|u_r\|_{L^2}^2\,.
\end{equation}
Combining \eqref{eq:pouru0}, \eqref{eq:I12def}, \eqref{eq:pouru1}, 
\eqref{eq:estimFur} and using Young's inequality, we obtain the 
preliminary estimate
\begin{equation}\label{eq:estimur0}
  \|\cA^\frac12 \partial_r^*u_r\|_{L^2}^2 +\|u_r\|_{L^2}^2 \,\le\, C_4\Bigl(
  \frac{k^2}{a^2(m^2+k^2)}+\frac{1}{am}\Bigr)\|u_r\|_{L^2}^2 + C_4 \Bigl(\frac{1}{a^2} 
  + \frac{1}{a^4}\Bigr) \|f\|_{L^2}^2\,,
\end{equation}
where the constant $C_4 > 0$ depends only on $\Omega$.   

If $ma \ge 2 C_4$, it is clear that \eqref{eq:estimur0} implies \eqref{eq:kpetit}. 
In the rest of the proof, we assume therefore that $ma \le 2C_4$. To obtain 
the improved bound \eqref{eq:kpetit}, the idea is to control the integral term $I_2$
in a different way. Denoting
\[
  Z(r) \,=\, -r\partial_r \Bigl(\frac{W(r)}{m^2{+}k^2r^2}\Bigr) \,>\, 0\,,
\]
we observe that
\begin{equation}\label{eq:I2exp}
  I_2 \,=\,  -\int_0^\infty \frac{im}{\gamma}\,Z(r)|u_r|^2r\dd r \,=\,
  \int_0^\infty \frac{m^2(b-\Omega) -iam}{|\gamma|^2}\,Z(r)|u_r|^2r\dd r\,.
\end{equation}
As $\Omega(r) \le 1$ for all $r$, a lower bound on $\Re I_2$ is obtained
if we replace $b-\Omega$ by $b -1$ in \eqref{eq:I2exp}. Thus, taking 
the real part of \eqref{eq:pouru0}, we obtain the bound
\begin{equation}\label{eq:Ireal}
 \|\cA^\frac12 \partial_r^*u_r\|_{L^2}^2 +\|u_r\|_{L^2}^2 + (b-1)
 \int_0^\infty \frac{m^2}{|\gamma|^2}\,Z(r)|u_r|^2r\dd r \,\le\, |I_1| 
 + |I_3|\,,
\end{equation}
where $I_3 = \int_0^\infty \cF(r)\bar{u}_r r\dd r$. If $b \ge 1$, we 
can drop the integral in the left-hand side, and using the estimates
\eqref{eq:pouru1}, \eqref{eq:estimFur} on $|I_1|$, $|I_3|$ we arrive 
at \eqref{eq:kpetit}. If $b < 1$, we consider also the imaginary part 
of \eqref{eq:pouru0}, which gives the inequality
\begin{equation}\label{eq:Iimag}
  \int_0^\infty \frac{am}{|\gamma|^2}\,Z(r)|u_r|^2r\dd r \,\le\, |I_1| + |I_3|\,.
\end{equation}
Combining \eqref{eq:Ireal}, \eqref{eq:Iimag} so as to eliminate the 
integral term, we obtain
\begin{equation}\label{eq:Itotal}
  \|\cA^\frac12 \partial_r^*u_r\|_{L^2}^2 + \|u_r\|_{L^2}^2 \,\le\, 
  \Bigl(1 + \frac{m(1-b)}{a}\Bigr)\bigl(|I_1| + |I_3|\bigr)\,.
\end{equation}
If $b \ge -1$, then $m(1-b)/a \le 2m/a \le 4 C_4/a^2$. If $b \le -1$, we 
can assume that $m(1-b) \le 4 C_1$, because in the converse case we have
\[
  |\gamma(r)| \,\ge\, m(\Omega(r) - b) \,\ge\, -mb \,\ge\, \frac{m(1-b)}{2}
  \,\ge\, 2 C_1\,, \quad \hbox{for all } r > 0\,,
\]
so that we can apply Lemma~\ref{lem:grandgamma} with $M = 2 C_1$ and
deduce \eqref{eq:kpetit} from \eqref{eq:gg} and \eqref{eq:estimur0}.
So, in all relevant cases, the right-hand side of \eqref{eq:Itotal} is
smaller than $C(1+a^{-2}) \bigl(|I_1| + |I_3|\bigr)$, and using the
estimates \eqref{eq:pouru1}, \eqref{eq:estimFur} on $|I_1|$, $|I_3|$
we obtain \eqref{eq:kpetit}.
\end{proof}

\begin{rem}\label{rem:Adr}
Estimate \eqref{eq:estimur0} implies in particular that
\begin{equation}\label{eq:borneh1}
  \|\cA^\frac12 \partial_r^* u_r\|_{L^2}^2 \,\le\, C_4\Bigl(\frac{1}{a} 
  + \frac{1}{a^2}\Bigr)\,\|u_r\|_{L^2}^2 + C_4 \Bigl(\frac{1}{a^2} 
  + \frac{1}{a^4}\Bigr)\|f\|_{L^2}^2\,.
\end{equation}
In view of Corollary~\ref{cor:uthetaz}, this shows that controlling the 
quantity $\|u_r\|_{L^2}$ in terms of $\|f\|_{L^2}$ is equivalent to 
the full resolvent estimate, because the azimuthal and vertical velocities 
can be estimated using \eqref{eq:uthetaz}, \eqref{eq:borneh1}. 
As an aside, we also observe that \eqref{eq:estimur0} provides an 
explicit resolvent estimate if $a > 0$ is sufficiently large, 
for instance if $a \ge 2 C_4 +1$. Thus we may assume in the sequel 
that $a$ is bounded from above by a constant depending only on $\Omega$.  
\end{rem}

To estimate the radial velocity $u_r$ in the regime where $k$ is 
large compared to $m$, it is convenient to introduce the auxiliary 
function $v(r) = \gamma(r)^{-1/2} u_r(r)$ (this idea already used 
in \cite{GS1} is borrowed from \cite{HG}). The new function $v$ 
satisfies the differential equation
\begin{equation}\label{eq:pourv}
  -\partial_r\bigl(\cA(r)\gamma(r)\partial_r^* v\bigr) + \cE(r) v \,=\, 
  \gamma(r)^{1/2} \cF(r)\,, \qquad r > 0\,,
\end{equation}
where $\cA(r)$, $\cF(r)$ are as in \eqref{eq:eigscalarf} and 
\begin{equation*}
  \cE(r) \,=\, \gamma(r) + \frac{k^2}{\gamma(r)}\,\cA(r)\Phi(r)
  + \frac{imr}{2}\partial_r\Bigl(\frac{W(r)+2\Omega(r)}{m^2+k^2r^2}\Bigr)
  - \frac{m^2\Omega'(r)^2}{4\gamma(r)}\,\cA(r)\,.
\end{equation*}

\begin{lem}\label{lem:HG12}
If $m \ge 1$ and $u,f\in X_{m,k}$ satisfy \eqref{eq:eigsysf}, there exists 
a constant $C_5 > 0$ depending only on $\Omega$ such that the function 
$v(r)= \gamma(r)^{-1/2}u_r(r)$ satisfies the estimate
\begin{equation}\label{eq:estimv}
  \|\cA^{1/2}\partial_r^* v\|_{L^2}^2 + \|v\|_{L^2}^2 \,\le\, \frac{C_5}{a^2}\,
  \frac{m^2}{m^2+k^2}\,\|1_B v\|_{L^2}^2 + C_5\Bigl(\frac{1}{a^3}+\frac{1}{a^5}
  \Bigr)\|f\|_{L^2}^2\,, 
\end{equation}
where $1_B$ is the indicator function of the set $B = \{r > 0\,;\, 
|\gamma(r)| \le r|\Omega'(r)|\}$. 
\end{lem}

\begin{proof}
Multiplying both sides of \eqref{eq:pourv} by $r \bar v$, integrating 
the resulting expression over $\Rp$ and taking the real part, we obtain 
the identity
\begin{equation}\label{eq:pourv0}
 a \int_0^\infty \biggl\{\cA |\partial_r^* v|^2 + \Bigl(1 + \frac{k^2}{
 |\gamma|^2}\cA\Phi\Bigr)|v|^2\biggr\}r\dd r =  \Re\int_0^\infty \!\!\bar{v} 
 \gamma^\frac12 \cF r\dd r + \frac{a}{4}\int_0^\infty \!\!m^2\Omega'^2
 \frac{\cA}{|\gamma|^2}|v|^2r\dd r\,.
\end{equation}
Since $\Phi\ge 0$, the left-hand side of \eqref{eq:pourv0} is bounded
from below by $a\bigl(\|\cA^\frac12\partial_r^*v\|_{L^2}^2 + \|v\|_{L^2}^2\bigr)$.
On the other hand, repeating the proof of Lemma~\ref{lem:estimF}, we 
can estimate the first integral in the right-hand side as follows\:
\begin{equation}\label{eq:pourv7}
\begin{split}
  \Bigl|\int_0^\infty \bar{v} \gamma^\frac12 \cF r\dd r\Bigr| \,&\le\, 
  \frac{2}{a^{1/2}}\,\|\cA^\frac12 \partial_r^*v\|_{L^2}\|f\|_{L^2} + C 
  \Bigl(\frac{1}{a^{1/2}} + \frac{1}{a^{3/2}}\Bigr)\|v\|_{L^2}\|f\|_{L^2} \\
  \,&\le\, \frac{a}{2}\bigl(\|\cA^\frac12\partial_r^*v\|_{L^2}^2+\|v\|_{L^2}^2\bigr) 
  +  C\Bigl(\frac{1}{a^2} + \frac{1}{a^4}\Bigr)\|f\|_{L^2}^2\,,
\end{split}
\end{equation}
where the constant $C > 0$ depends only on $\Omega$. It remains to
estimate the second integral in the right-hand side of \eqref{eq:pourv0}. 
Defining $\cG(r) = m^2\Omega'(r)^2\frac{\cA(r)}{|\gamma(r)|^2}$, we observe
that 
\begin{equation}\label{eq:pourv8}
\begin{split}
  \frac{a}{4}\int_0^\infty m^2\Omega'^2\frac{\cA}{|\gamma|^2}\,|v|^2r\dd r 
  \,&\le\, \frac{a}{4}\,\|1_{\{\cG\le 1\}}v\|_{L^2}^2 + \frac{a}{4}\,\|\cG\|_{L^\infty} 
  \|1_{\{ \cG\ge 1\}} v\|_{L^2}^2 \\ 
  \,&\le\, \frac{a}{4}\,\|v\|_{L^2}^2 + \frac{1}{4a}\,\frac{m^2}{m^2+k^2}\, 
  \|(1{+}r^2){\Omega'}^2\|_{L^\infty} \|1_{\{ \cG\ge 1\}} v\|_{L^2}^2\,, 
\end{split}
\end{equation}
where the upper bound on the quantity $\|\cG\|_{L^\infty}$ is obtained using the 
definition of $\cA$ and the fact that $|\gamma(r)| \ge \Re(s) = a$. Now, 
if $\cG(r) \ge 1$, then $|\gamma(r)|^2 \le m^2\Omega'(r)^2 \cA(r) \le r^2 
\Omega'(r)^2$, so that the set $\{\cG\ge 1\}$ is contained in $B = \{r > 0\,;\, 
|\gamma(r)| \le r|\Omega'(r)|\}$. Thus, combining \eqref{eq:pourv0}, 
\eqref{eq:pourv7}, and \eqref{eq:pourv8}, we obtain \eqref{eq:estimv}.
\end{proof}

The following result is a rather direct consequence of 
Lemmas~\ref{lem:grandgamma} and \ref{lem:HG12}\:

\begin{lem}\label{lem:ksmg}
If $m \ge 1$ and $u,f \in X_{m,k}$ satisfy \eqref{eq:eigsysf}, there 
exists a constant $C_6 > 0$, depending only on $\Omega$, such that 
the inequality
\begin{equation}\label{eq:ksmg0}
 \|u\|_{L^2} \,\le\, C_6\Bigl(\frac{1}{a} + \frac{1}{a^{7/2}}\Bigr)\|f\|_{L^2} 
 \end{equation}
holds in each of the following three situations\:
\[
  \hbox{i) } ak \ge C_6 m\,, \quad~
  \hbox{ii) } am \ge C_6 \hbox{ and } C_6(1-b) \le a\,, \quad~
  \hbox{iii) } am \ge C_6 \hbox{ and } C_6b \le a\,.
\]
\end{lem}

\begin{proof}
Applying Lemma~\ref{lem:grandgamma} with $M = 3 C_1$, we deduce from 
\eqref{eq:gg} that
\begin{equation}\label{eq:ksmg1}
  \|1_{\{|\gamma|\ge M\}}u\|_{L^2} \,\le\, \frac12\|1_{\{|\gamma|\le M\}}u\|_{L^2} 
  + \frac12\|f\|_{L^2}\,.
\end{equation}
In the sequel, we may thus focus our attention to the region where 
$|\gamma|\le M$. Our strategy is to use Lemma~\ref{lem:HG12}, which
requires a good control on the term involving $1_Bv$ in the 
right-hand side of \eqref{eq:estimv}. We consider three cases separately. 

\smallskip
\noindent i) If $ak \ge m\sqrt{2C_5}$, we simply observe that
\begin{equation}\label{eq:ksmg2}
  \frac{C_5}{a^2}\,\frac{m^2}{m^2+k^2}\,\|1_B v\|_{L^2}^2 \,\le\,  
  \frac12 \|1_B v\|_{L^2}^2 \,\le\, \frac12 \|v\|_{L^2}^2\,.
\end{equation}

\noindent ii) By definition, for any $r > 0$, we have  
\begin{equation}\label{eq:Bdomain}
  r \in B \quad \hbox{if and only if} \quad a^2 + m^2(\Omega(r)-b)^2 \,\le\, 
  r^2 \Omega'(r)^2\,.
\end{equation}
Clearly $B = \emptyset$ if $a > \|r\Omega'\|_{L^\infty}$, hence we 
may assume that $a \le \|r\Omega'\|_{L^\infty}$. Since $\Omega'(r) 
= \cO(r)$ as $r \to 0$ by assumption H1, there exists a small constant
$\epsilon > 0$ (depending only on $\Omega$) such that inequality 
\eqref{eq:Bdomain} cannot be satisfied if $r \le \epsilon a^{1/2}$. 
On the other hand, if $r \ge \epsilon a^{1/2}$, then $\Omega(r) \le 
\Omega(\epsilon a^{1/2}) \le 1 - 2\delta a$ for some sufficiently small 
$\delta > 0$. Thus, if we assume that $b \ge 1-\delta a$ and $m\delta a 
>  \|r\Omega'\|_{L^\infty}$, we see that $m(b-\Omega(r)) \ge m\delta a 
> \|r\Omega'\|_{L^\infty}$, so that  inequality \eqref{eq:Bdomain} is not 
satisfied either. Summarizing, we have $B = \emptyset$ if $ma \ge C$ 
and $C(1-b) \le a$ for some sufficiently large $C > 0$. 

\smallskip
\noindent iii) Similarly, since $r\Omega'(r) = \cO(r^{-2})$ as
$r \to \infty$ by assumption H1, there exists a large constant 
$\rho > 0$ (depending only on $\Omega$) such 
that \eqref{eq:Bdomain} cannot be satisfied if
$r \ge \rho a^{-1/2}$. If $r \le \rho a^{-1/2}$, we have
$\Omega(r) \ge \Omega(\rho a^{-1/2}) \ge 2\sigma a$ for some
$\sigma > 0$. Thus, if we assume that $b \le \sigma a$ and
$m\sigma a > \|r\Omega'\|_{L^\infty}$, inequality \eqref{eq:Bdomain}
is never satisfied, so that $B = \emptyset$.

\smallskip
In all three cases, we deduce from \eqref{eq:estimv} the estimate
\begin{equation}\label{eq:estimvbis}
  \|\cA^{1/2}\partial_r^* v\|_{L^2}^2 + \frac12 \|v\|_{L^2}^2 \,\le\, 
  C_5 \Bigl(\frac{1}{a^3}+\frac{1}{a^5}\Bigr)\|f\|_{L^2}^2\,. 
\end{equation}
As $u_r(r) = \gamma(r)^{1/2}v(r)$, we have $\|1_{\{|\gamma|\le M\}}u_r\|_{L^2} 
\le M^{1/2} \|1_{\{|\gamma|\le M\}}v\|_{L^2} \le M^{1/2} \|v\|_{L^2}$, and
\[
  \bigl\|1_{\{|\gamma|\le M\}} \cA^{\frac12}\partial_r^*u_r\bigr\|_{L^2} \,\le\, 
  M^{1/2}  \bigl\|1_{\{|\gamma|\le M\}} \cA^{\frac12}\partial_r^*v\bigr\|_{L^2} + 
  \frac{\|r\Omega'\|_{L^\infty}}{2a^{1/2}}\,\|v\|_{L^2}\,.
\]
Thus, using the representations \eqref{eq:uthetaexp}, \eqref{eq:uzexp} 
of the azimuthal and vertical velocities, we deduce from 
\eqref{eq:estimvbis} that
\[
  \|1_{\{|\gamma|\le M\}}u_r\|_{L^2} + \|1_{\{|\gamma|\le M\}}u_\theta\|_{L^2} 
  + \|1_{\{|\gamma|\le M\}}u_z\|_{L^2} \,\le\, C\Bigl(\frac{1}{a}
  +\frac{1}{a^{7/2}}\Bigr)\|f\|_{L^2}\,. 
\]
Finally, invoking \eqref{eq:ksmg1} to bound $\|1_{\{|\gamma|\ge M\}}u\|_{L^2}$ 
in terms of $\|1_{\{|\gamma|\le M\}}u\|_{L^2}$, and recalling that we can assume 
$a \le 2C_4 +1$ by Remark~\ref{rem:Adr}, we arrive at \eqref{eq:ksmg0}. 
\end{proof}

\begin{rem}\label{rema:alter}
Alternatively, one can obtain the resolvent estimate in case iii) 
by the following argument. If $m \ge 1$ is large and $b > 0$ is small, 
the inequality $|\gamma(r)| \le M := 3C_1$ can be satisfied only 
if $r \gg 1$. In that region, the coefficients $\Omega(r)$ and 
$W(r)$ in \eqref{eq:eigsysf} are very small, and so is the 
pressure $p$ in view of Lemma~\ref{lem:pressionloin}. It is 
thus easy to estimate $\|1_{\{|\gamma| \le M\}}u\|_{L^2}$ in 
terms $\|f\|_{L^2}$ directly from \eqref{eq:eigsysf}. Combining 
this observation with Lemma~\ref{lem:grandgamma} gives the desired result. 
\end{rem}

\subsection{End of the proof of  Proposition~\ref{prop:main}}
\label{sec43}

If we combine Lemma~\ref{lem:axi}, Lemma~\ref{lem:kpetit},
Remark~\ref{rem:Adr}, and Lemma~\ref{lem:ksmg}, we obtain the
following statement which specifies the regions in the parameter space
where we could obtain a uniform resolvent estimate, with explicit (or
at least computable) constant.

\begin{cor}\label{cor:explicit}
Assume that $m \in \N$, $k \ge 0$, and $s \in \C$ with $\Re(s) = a > 0$. 
There exists a constant $C > 0$, depending only on $\Omega$, such that 
the resolvent estimate 
\begin{equation}\label{eq:unifresksmg}
  \bigl\|(s - L_{m,k})^{-1}\bigr\|_{X_{m,k} \to X_{m,k}} \,\le\, 
  C\Bigl(\frac{1}{a}+\frac{1}{a^4}\Bigr)
\end{equation}
holds in each of the following cases\:
\begin{equation}\label{eq:6cases}
\begin{array}{lll}
  1)~m = 0\,, & 2)~a \ge C\,, & 3)~ma^2 \ge Ck\,, \\[1mm]
  4)~ak \ge C m\,,\quad &5)~am \ge C \hbox{ and }C(1-b) \le a\,, 
  \quad &6)~am \ge C \hbox{ and } Cb \le a\,.
\end{array}
\end{equation}
We recall that $b$ is defined by \eqref{eq:sdef} when $m \neq 0$. 
\end{cor}

To conclude the proof of Proposition~\ref{prop:main}, we use 
a contradiction argument to establish a resolvent estimate 
in the regions that are not covered by Corollary \ref{cor:explicit}. 
More precisely, if we consider a sequence of values of the 
parameters $m,k,s$ (with $\Re(s) = a$) such that none of 
the conditions 1)--6) in \eqref{eq:6cases} is satisfied, 
two possibilities can occur. Either the angular Fourier mode 
$m$ goes to infinity, as well as the vertical wave number $k$, 
and the parameter $b$ remains in the interval $[a/C,1{-}a/C] 
\subset (0,1)$. In that case, after extracting a subsequence, we 
can assume that $b$ converges to some limit. So, to establish 
the resolvent estimate, we have to prove that, for any
$b \in (0,1)$,
\begin{equation}\label{eq:remaining1} 
  \sup_{\Re(s) = a} ~\limsup_{\substack{m\to +\infty,\\ \Im(s)/m \to -b}} 
  \bigl\|(s - L_{m,k})^{-1}\bigr\|_{X_{m,k} \to X_{m,k}} \,<\, \infty\,.
\end{equation}
The other possibility is that the angular Fourier mode $m \ge 1$
stays bounded, as well as the vertical wave number $k \ge k_0 := 
a^2/C$. In that case, we have to prove that, for all $N \ge 1$,
\begin{equation}\label{eq:remaining2}
  \sup_{\Re(s) = a} ~\sup_{\substack{1 \le m \le N,\\ k_0 \le k \le N}}
  ~\bigl\|(s - L_{m,k})^{-1}\bigr\|_{X_{m,k} \to X_{m,k}} \,<\, \infty\,.
\end{equation}

\noindent{\bf Proof of estimate \eqref{eq:remaining1}\:} \\
We argue by contradiction and assume the existence of sequences
$(m_n)_{n\in \N}$ in $\N$, $(k_n)_{n\in \N}$ in $\R_+$, $(b_n)_{n\in \N}$ 
in $\R$ and $(u^n)_{n\in \N}$, $(f^n)_{n\in \N}$ in $X_{m_n,k_n}$ with the 
following properties\: $u^n,f^n$ are solutions of the resolvent system 
$(s_n-L_{m_n,k_n})u^n=f^n$ where $s_n=a-im_nb_n$, $\|u^n\|_{L^2}=1$ 
$\forall n \in \N$, and we have $\|f^n\|_{L^2} \to 0$, $m_n \to +\infty$, 
and $b_n \to b$ as $n\to +\infty$. Without loss of generality we may 
assume that $b_n \in (0,1)$ for all $n\in \N$, and we define $r_n = 
\Omega^{-1}(b_n)$; in particular $r_n \to \bar r := \Omega^{-1}(b)$
as $n \to +\infty$. We also denote by $(p_n)_{n\in \N}$ the sequence
of pressures associated to $u^n$, namely $p_n=P_{m_n,k_n}[u^n]$, and
we set $\gamma_n(r) = a+im_n(\Omega(r)-b_n)$. 

In view of inequalities \eqref{eq:uthetaz} and \eqref{eq:borneh1}, 
the normalization condition $\|u^n\|_{L^2} = 1$ and the assumption 
that $\|f^n\|_{L^2} \to 0$ as $n \to \infty$ imply that the quantity
$\|u_r^n\|_{L^2}$ is bounded from below for large values of $n$, 
namely
\begin{equation}\label{eq:urdomine}
  I_r \,:=\, \liminf_{n\to +\infty} \|u^n_r\|_{L^2}^2 \,>\, 0\,.
\end{equation}
Setting $M = C_1\sqrt{2/I_r}$, we deduce from \eqref{eq:urdomine} and 
Lemma~\ref{lem:grandgamma} that
\begin{equation}\label{eq:urdomineloc}
  \liminf_{n\to +\infty} \int_{\{|\gamma_n|\le M\}}|u^n_r(r)|^2\,r\dd r \,\ge\, 
  \frac{I_r}{2} \,>\, 0\,.
\end{equation}
As the angular velocity $\Omega$ is continuously differentiable and
strictly decreasing on $\Rp$, the set $\{|\gamma_n|\le M\}$ is asymptotically
contained in the interval $[r_n-R/m_n,r_n+R/m_n]$, where $R > 0$ 
is a constant that depends only on $\Omega$ and $I_r$ (one may 
take $R = 2M |\Omega'(\rb)|^{-1}$). Since the length of that interval 
shrinks to zero as $n \to \infty$, it is useful to introduce rescaled vector
fields and functions by setting
\[
  u^n(r) = m_n^{1/2}\, \tilde u^n(m_n(r{-}r_n))\,, \quad
  f^n(r) =  m_n^{1/2}\, \tilde f^n(m_n(r{-}r_n))\,, \quad
  p_n(r) = m_n^{-1/2}\, \tilde p_n(m_n(r{-}r_n))\,.
\]
Note that the new variable $y := m_n(r{-}r_n)$ is defined on the $n$-dependent
domain $(-m_nr_n,\infty)$. Likewise, we set $\Omega(r)=\tilde 
\Omega_n(m_n(r{-}r_n))$, $W(r)=\tilde W_n(m_n(r{-}r_n))$ and $\gamma_n(r)=
\tilde \gamma_n(m_n(r{-}r_n))$.
The system \eqref{eq:eigsysf} may then be rewritten as 
\begin{equation}\label{eq:eigsysftilde1}
\begin{array}{l}
  \tilde \gamma_n(y) \tilde u^n_r - 2\tilde \Omega_n(y)\tilde u^n_\theta \,=\, 
  -\partial_y \tilde p_n + \tilde f^n_r\,, \\[1mm]
  \tilde  \gamma_n(y) \tilde u^n_\theta  + \tilde W_n(y)\tilde u^n_r \,=\, 
  -\frac{i}{r_n+y/m_n} \tilde p_n+\tilde f^n_\theta\,, \\[1mm]
  \tilde \gamma_n(y)\tilde u^n_z  \,=\, -i\frac{k_n}{m_n} \tilde p_n+\tilde f^n_z\,, 
\end{array} 
\end{equation}
and the incompressibility condition becomes
\begin{equation}\label{eq:eigsysftilde2}
  \partial_y \tilde u^n_r + \tfrac{i}{r_n+y/m_n} \tilde u^n_\theta + 
  i\tfrac{k_n}{m_n} \tilde u^n_z \,=\, -\tfrac{1}{r_nm_n+y}\tilde u^n_r\,.
\end{equation}
After this change of variables, inequality \eqref{eq:urdomineloc} implies 
the lower bound
\begin{equation}\label{eq:Urdomineloc}
  \liminf_{n\to +\infty} \int_{-R}^{R} |\tilde u^n_r(y)|^2\dd y \,\ge\, 
  \frac{I_r}{2\rb} \,>\, 0\,.
\end{equation}
Since, by assumption, inequalities 3) and 4) in \eqref{eq:6cases} are not 
satisfied, we can suppose without loss of generality that $k_n/m_n \to \delta \in 
(0,+\infty)$ as $m\to +\infty$. By construction, we also have 
$\tilde \Omega_n(y) \to \Omega(\bar r)$, $\tilde W_n(y) \to W(\bar r)$ and
$\tilde \gamma_n(y) \to \gamb(y) := a+i\Omega'(\bar r)y$ as $n\to +\infty$,
uniformly on any compact subset of $\R$.

Using the normalization condition for $u^n$, we observe that
\[
  1 \,=\, \int_{-m_n r_n}^{\infty} |\tilde u^n(y)|^2\Bigl(r_n+\frac{y}{m_n}\Bigr)
  \dd y \,\ge\, \frac{r_n}{2} \int_{-m_n\frac{r_n}{2}}^{\infty} |\tilde u^n(y)|^2\dd y\,.
\]
Extracting a subsequence if needed, we may therefore assume that
$\tilde u^n \rightharpoonup U$ in $L^2(K)$ for each compact subset
$K\subset \R$, where $U \in L^2(\R)$ and $\|U\|_{L^2}^2 \le 2/\rb$.
Similarly, using the uniform bounds on the pressure given by 
Lemma~\ref{lem:ell}, we may assume that $\tilde p_n \to P$ and 
$\partial_y \tilde p_n \rightharpoonup P'$ in $L^2(K)$, for each compact 
subset $K\subset \R$, where $P \in H^1_{\rm loc}(\R)$ and $P'\in L^2(\R)$. 
The radial velocities $\tilde u^n_r$ have even better convergence
properties. Indeed, it follows from \eqref{eq:borneh1} that the quantity
$\|\cA_n^{1/2}\partial_r^* u_r^n\|_{L^2}$ is uniformly bounded for $n$ 
large, and since $\|\cA_n^{1/2} r^{-1} u_r^n\|_{L^2} \le 1/m_n \to 0$ we deduce
that $\|\cA_n^{1/2}\partial_r u_r^n\|_{L^2}$ is uniformly bounded too. 
After the change of variables, this implies that
\[
  C \,\ge\, \int_{-m_n r_n}^{\infty} \frac{m_n^2 r^2}{m_n^2 + k_n^2 r^2}\,
  |\partial_y \tilde u_r^n(y)|^2\,r\dd y \,\ge\, \frac{r_n}{2}\,
  \frac{r_n^2}{4+\delta_n^2 r_n^2} \int_{-m_n\frac{r_n}{2}}^{\infty} 
  |\partial_y \tilde u_r^n(y)|^2\dd y\,,
\]
where $r = r_n+y/m_n$ and $\delta_n = k_n/m_n$. Thus $U_r \in H^1(\R)$, 
and extracting a further subsequence if necessary we can assume that
$\partial_y \tilde u^n_r \rightharpoonup U_r'$ and $\tilde u^n_r \to U_r$ 
in $L^2(K)$, for each compact subset $K\subset \R$. In particular, 
we deduce from \eqref{eq:Urdomineloc} that $U_r$ is not identically zero. 
Moreover, passing to the limit in \eqref{eq:eigsysftilde1}, 
\eqref{eq:eigsysftilde2}, we obtain the asymptotic system
\begin{equation}\label{eq:eigsyslim}
\begin{array}{l}
 (a+i\Omega'(\rb)y) U_r - 2 \Omega(\rb) U_\theta \,=\, -P'\,,\\[1mm]
 (a+i\Omega'(\rb)y) U_\theta  + W(\rb) U_r \,=\, -\frac{i}{\rb} P\,, \\[1mm]
 (a+i\Omega'(\rb)y) U_z  \,=\, -i\delta P\,, 
\end{array} \qquad\quad
  U_r' + \tfrac{i}{\rb} U_\theta + i\delta U_z \,=\, 0\,,
\end{equation}
where equalities hold almost everywhere. We claim that system
\eqref{eq:eigsyslim} does not possess any solution such that
$U \in L^2_{\rm loc}(\R)$, $P \in H^1_{\rm loc}(\R)$ and such that
$U_r \in H^1(\R)$ is nontrivial. This will provide the desired
contradiction.

Indeed, if we repeat the proof of Lemma~\ref{lem:resODE} (with $f = 0$), 
we can extract from system \eqref{eq:eigsyslim} a second-order 
differential equation for the radial velocity $U_r$. Eliminating 
the pressure $P$ and the azimuthal velocity $U_\theta$, we obtain as 
in \eqref{eq:ODE2}, \eqref{eq:ODE4}\:
\[
  \delta\Bigl(U_r' - \frac{iW(\rb)}{\rb\,\gamb(y)}U_r\Bigr) + 
  i \Bigl(\delta^2 + \frac{1}{\rb^2}\Bigr)U_z \,=\,0\,, \qquad
  U_z' + \frac{iW(\rb)}{\rb\,\gamb(y)}U_z -i\delta
  \Bigl(1 + \frac{\Phi(\rb)}{\gamb(y)^2}\Bigr) \,=\, 0\,, 
\]
and combining these relations we arrive at
\begin{equation}\label{eq:lim3}
 -U_r'' + \biggl[\Bigl(\delta^2 + \frac{1}{\rb^2}\Bigr) + 
  \frac{\Phi(\rb)\delta^2}{\gamb(y)^2}\Biggr] U_r \,=\, 0\,, 
  \qquad y \in \R\,,
\end{equation}
where
$\Phi(\rb) = 2 \Omega(\rb)W(\rb) > 0$. If we observe that $\gamb(y) = 
a + i\Omega'(\rb)y = i\Omega'(\rb)(y + ic)$, where  $c=-a/\Omega'(\rb)$, 
we can write \eqref{eq:lim3} in the equivalent form
\begin{equation}\label{eq:besseltype}
  -U_r'' + \Bigl(\kappa^2 - \frac{J(\rb)\delta^2}{(y+ic)^2} \Bigr) U_r 
  \,=\, 0\,,\qquad y \in \R\,,
\end{equation}
where  $\kappa^2 = 1/\rb^2+\delta^2$ and $J(\rb) = \Phi(\rb)/\Omega'(\rb)^2$. 
Up to a multiplicative constant, the unique solution of \eqref{eq:besseltype}
that belongs to $L^2(\Rp)$ is 
\begin{equation}\label{eq:Knu}
  U_r(y) \,=\, (y+ic)^{1/2} K_\nu\bigl(\kappa(y+ic)\bigr)\,, \qquad y \in \R\,, 
\end{equation}
where $K_\nu$ is the modified Bessel function, see \cite[Section~9.6]{AS}, 
and $\nu \in \C$ is determined, up to an irrelevant sign, by the
relation $\nu^2 = \frac14 - J(\rb)\delta^2$. In fact, any linearly
independent solution of \eqref{eq:besseltype} grows like
$\exp(\kappa y)$ as $y \to +\infty$. Now, it is well known that
the function $K_\nu(\kappa(y+ic))$ has itself an exponential
growth as $y \to -\infty$, see \cite[Section~9.7]{AS}, and this implies 
that \eqref{eq:besseltype} has no nontrivial solution in $L^2(\R)$.

\medskip\noindent{\bf Proof of estimate \eqref{eq:remaining2}\:} \\
This is the only place where we use our assumption H2 on the vorticity 
profile. According to Proposition~\ref{prop:GS1}, which is the main 
result of \cite{GS1}, the resolvent operator $(s - L_{m,k})^{-1}$ is 
well defined as a bounded linear operator in $X_{m,k}$ for any 
$m \in \N$, any $k \in \R$, and any $s \in \C$ with $\Re(s) \neq 0$. 
To prove \eqref{eq:remaining2}, it remains to show that, for any 
fixed $m$, the resolvent estimate holds uniformly in $k$ on compact 
subsets of $\Rp = (0,\infty)$, and uniformly in $s$ on vertical lines. Actually, 
we can assume that the spectral parameter lies in a compact set too, 
because if $m$ is fixed and $|\Im(s)| \ge m + 2C_1$, we have 
$|\gamma(r)| \ge |\Im(s)| - m \ge 2C_1$ and the resolvent bound 
follows from estimate \eqref{eq:gg} with $M = 2C_1$. So the only 
missing step is\:

\begin{lem}\label{lem:localbd}
For any $m \in \Z$, the resolvent norm $\|(s - L_{m,k})^{-1}
\|_{X_{m,k} \to X_{m,k}}$ is uniformly bounded in the neighborhood 
of any point $(k,s) \in \R \times \C$ with $k \neq 0$ and 
$\Re(s) > 0$. 
\end{lem}

\begin{proof}
Since the function space  $X_{m,k}$ changes when $k$ is varied, 
due to the incompressibility condition, the result does not immediately 
follow from standard perturbation theory. However, it is easy to 
reformulate the problem so that perturbation theory can be applied. 
It is sufficient to note that, for any fixed $k^*\neq 0$, the mappings
\[
  M_k : X_{m,k^*} \to  X_{m,k}\,, \qquad  \bigl(u_r,u_\theta,u_z\bigr) 
  \,\mapsto\, \bigl(u_r,u_\theta,\frac{k^*}{k}u_z\bigr)\,,
\]
are linear homeomorphisms that depend continuously on $k$ in a 
neighborhood of $k^*$. Given $s \in \C$, $m \in \Z$, and $k \in \R$ 
close to $k^*$, the resolvent equation $(s - L_{m,k})u = f$ for 
$u,f \in X_{m,k}$ is equivalent to the conjugated equation
$(s - \cL_{m,k})v = g$, where $u = M_k v$, $f = M_k g$, and 
\begin{equation}\label{eq:conjugL}
 \cL_{m,k} \,=\, M_k^{-1} L_{m,k} M_k \,:\, X_{m,k^*} \to  X_{m,k^*}\,.  
\end{equation}
Now, using in particular estimate \eqref{eq:deltap} in
Lemma~\ref{lem:pcont}, it is straightforward to verify that the
operator $\cL_{m,k}$ depends continuously on $k$ as a bounded linear
operator in $X_{m,k^*}$, as long as $k \neq 0$. This implies that the
resolvent norm $\|(s - \cL_{m,k})^{-1}\|_{X_{m,k^*} \to X_{m,k^*}}$
depends continuously on the parameters $s$ and $k$, when $k$ stays in
a neighborhood of $k^*$, and the conclusion easily follows.
\end{proof}

\section{Appendix\: analysis in $X_{m,k}$}
\label{sec5}

We collect here various auxiliary results that are useful for our
analysis in Section~\ref{sec3}. We first show that smooth and
compactly supported divergence-free vector fields are dense in the
space $X_{m,k}$ defined by \eqref{eq:Xmkdef}, and we give simple
criteria for compactness in that space. Finally, we establish explicit
representations formulas for the pressure $p$ satisfying \eqref{eq:pmk}.

\subsection{Approximation in $X_{m,k}$}\label{sec51}

Truncating divergence-free vector fields is not straightforward, and 
a general solution to that problem involves the so-called Bogovskii 
operator, see e.g. \cite{Ga}. However, in the particular case of the 
space $X_{m,k}$ introduced in \eqref{eq:Xmkdef}, localization can
be performed in a rather elementary way, which we now describe. 

\begin{lem}\label{lem:truncation}
For any $m \in \Z$ and any $k \in \R$, the set of all $u \in X_{m,k}$ 
with compact support in $(0,+\infty)$ is dense in $X_{m,k}$. 
\end{lem}

\begin{proof}
Let $\phi,\psi : \R_+ \to \R$ be smooth, monotonic functions such that
\[
  \phi(r) \,=\, \begin{cases} 0 &\hbox{if } r \le \frac12\,, \cr
  1 &\hbox{if } r \ge 1\,, \end{cases} \qquad \hbox{and}\quad
  \psi(r) \,=\, \begin{cases} 1 &\hbox{if } r \le 1\,, \cr
  0 &\hbox{if } r \ge 2\,. \end{cases} \qquad
\]
Given $\epsilon \in (0,1)$, we define $\chi_\epsilon(r) = \min\{\phi(r/\epsilon),
\psi(\epsilon r)\}$. By construction $\chi_\epsilon$ is smooth and satisfies
$\chi_\epsilon(r) = 0$ if $r \le \epsilon/2$ or $r \ge 2/\epsilon$,
and $\chi_\epsilon(r) = 1$ if $\epsilon \le r \le 1/\epsilon$. 

Assume first that $m \neq 0$. Given $u \in X_{m,k}$, we define 
$v_\epsilon = u \chi_\epsilon + w_\epsilon e_\theta$, where
\[
  w_\epsilon(r) \,=\, \frac{i}{m}\,r \chi_\epsilon'(r) u_r(r)\,, 
  \quad r > 0\,.
\]
The corrector $w_\epsilon$ is tailored so that $\div v_\epsilon = 
(\div u) \chi_\epsilon + u_r \chi_\epsilon' + \frac{im}{r} w_\epsilon = 0$. 
Moreover $w_\epsilon$ is supported in the set $[\epsilon/2,2/\epsilon]$ 
by construction. Since $\chi_\epsilon(r) \to 1$ as $\epsilon \to 0$ for 
any $r > 0$,  it is clear that $\|u \chi_\epsilon - u\|_{L^2} \to 0$ as 
$\epsilon \to 0$. Moreover
\begin{align*}
  \|w_\epsilon\|_{L^2}^2 \,&=\, \frac{1}{m^2}\int_{\epsilon/2}^{\epsilon}
  \frac{r^2}{\epsilon^2}\,|\phi'(r/\epsilon)|^2\,|u_r(r)|^2 r\dd r + 
  \frac{1}{m^2}\int_{1/\epsilon}^{2/\epsilon} \epsilon^2 r^2 
  |\psi'(\epsilon r)|^2 \,|u_r(r)|^2 r\dd r \\
  \,&\le\, \frac{C}{m^2}\Bigl(\int_0^{\epsilon} |u_r(r)|^2 r\dd r 
  + \int_{1/\epsilon}^\infty |u_r(r)|^2 r\dd r\Bigr) 
  \,\xrightarrow[\epsilon \to 0]{}\, 0\,.
\end{align*}
Thus $\|v_\epsilon - u\|_{L^2} \to 0$ as $\epsilon \to 0$, which is 
the desired result. 

Next we assume that $m = 0$ and $k \neq 0$. Given any $u \in X_{0,k}$, the 
divergence-free condition $\partial_r^* u_r + ik u_z = 0$ implies that
\begin{equation}\label{urrep}
  u_r(r) \,=\, -\frac{ik}{r}\int_0^r u_z(s) s\dd s\,, \qquad \hbox{hence}
  \quad |u_r(r)|^2 \,\le\, \frac{k^2}{2} \int_0^r |u_z(s)|^2 s\dd s\,,
\end{equation}
for any $r > 0$. We now define $\tilde v_\epsilon = u \chi_\epsilon + \tilde 
w_\epsilon e_z$, where
\[
  \tilde w_\epsilon(r) \,=\, \frac{i}{k}\,\chi_\epsilon'(r) u_r(r)\,, 
  \quad r > 0\,.
\]
As before $\tilde v_\epsilon$ is divergence-free and supported in 
$[\epsilon/2,2/\epsilon]$. Moreover, using \eqref{urrep}, we find
\begin{align*}
  \|\tilde w_\epsilon\|_{L^2}^2 \,&=\, \frac{1}{k^2}\int_{\epsilon/2}^{\epsilon}
  \frac{1}{\epsilon^2}\,|\phi'(r/\epsilon)|^2\,|u_r(r)|^2 r\dd r + 
  \frac{1}{k^2}\int_{1/\epsilon}^{2/\epsilon} \epsilon^2  
  |\psi'(\epsilon r)|^2 \,|u_r(r)|^2 r\dd r \\
  \,&\le\, C \int_0^{\epsilon} |u_z(r)|^2 r\dd r + \frac{C \epsilon^2}{k^2}
  \int_{1/\epsilon}^\infty |u_r(r)|^2 r\dd r \,\xrightarrow[\epsilon \to 0]{}\, 0\,,
\end{align*}
and this shows that $\|\tilde v_\epsilon - u\|_{L^2} \to 0$ as $\epsilon \to 0$. 

Finally, if $u \in X_{0,0}$, the divergence-free condition asserts that 
$\partial_r^* u_r = 0$, hence $u_r = 0$. It follows that $u\chi_\epsilon$ 
is divergence-free, and we know that $\|u \chi_\epsilon - u\|_{L^2} \to 0$ as 
$\epsilon \to 0$. 
\end{proof}

Using Lemma~\ref{lem:truncation} and a standard regularization
procedure, we obtain:

\begin{prop}\label{prop:approximation}
For any $m \in \Z$ and any $k \in \R$, the set of all smooth, 
divergence-free vector fields with compact support in $(0,+\infty)$ 
is dense in $X_{m,k}$. 
\end{prop}

\begin{proof}
According to Lemma~\ref{lem:truncation}, it is sufficient to prove 
that any $u \in X_{m,k}$ with compact support can be approximated 
by smooth, divergence-free and compactly supported vector fields. 
Assume thus that $u \in X_{m,k}$ is such that $u(r) = 0$ for 
$r \le r_1$ and $r \ge r_2$, with $0 < r_1 < r_2 < \infty$. We 
consider the vector field $U = (U_1,U_2,U_3)$ in $\R^3$ defined 
by
\begin{equation}\label{eq:urep1}
  U(r\cos\theta, r\sin\theta,z) \,=\, \Bigl( u_r(r) e_r(\theta) + 
  u_\theta(r) e_\theta(\theta) + u_z(r) e_z\Bigr) \,e^{im\theta} \,e^{ikz}\,,
\end{equation}
where $r > 0$, $\theta \in \R/(2\pi\Z)$, and $z \in \R$. Then 
$\div U = 0$ and, for any fixed $x_3 \in \R$, the map $(x_1,x_2) 
\mapsto U(x_1,x_2,x_3)$ belongs to $L^2(\R^2,\C^3)$, because 
$\|U(\cdot,\cdot,x_3)\|_{L^2(\R^2)}^2 = 2\pi \|u\|_{L^2}^2 < \infty$.  
Given $\epsilon > 0$, we define the approximation
\[
  U^\epsilon(x_1,x_2,x_3) \,=\, \frac{1}{\epsilon^2}\int_{\R^2} 
  \chi\Bigl(\frac{x_1-y_1}{\epsilon},\frac{x_2-y_2}{\epsilon}\Bigr)
  \,U(y_1,y_2,x_3)\dd y_1 \dd y_2\,, 
\]
where $\chi : \R^2 \to \R_+$ is smooth, {\em radially symmetric}, 
supported in the unit ball, and normalized so that $\int \chi \dd x_1 
\dd x_2 = 1$. By construction, the vector field $U^\epsilon$ is smooth, 
divergence-free, and close to $U$ in the sense that $\|U^\epsilon(\cdot,
\cdot,x_3) - U(\cdot,\cdot,x_3)\|_{L^2(\R^2)} \to 0$ as $\epsilon \to 0$ 
for any $x_3 \in \R$. If $\epsilon \le r_1/2$, we also have 
$U^\epsilon(x_1,x_2,x_3) = 0$ whenever $r := (x_1^2 + x_2^2)^{1/2} \le 
r_1/2$ or $r \ge r_1 + r_2$. Under this assumption, since $\chi$ 
is radially symmetric, we can represent $U^\epsilon$ as
\begin{equation}\label{eq:urep2}
  U^\epsilon(r\cos\theta, r\sin\theta,z) \,=\, \Bigl(u^\epsilon_r(r) 
  e_r(\theta) + u^\epsilon_\theta(r) e_\theta(\theta) + u^\epsilon_z(r) 
  e_z\Bigr) \,e^{im\theta} \,e^{ikz}\,,
\end{equation}
for some {\em smooth} vector field $u^\epsilon = u_r^\epsilon e_r + 
u_\theta^\epsilon e_\theta + u_z^\epsilon e_z \in X_{m,k}$, which is 
supported in the compact interval $[r_1/2,r_2 + r_1] \subset (0,\infty)$. 
Here the condition on the support is essential, because the 
unit vectors $e_r, e_\theta$ are smooth only away from the axis $r = 0$. 
From \eqref{eq:urep1}, \eqref{eq:urep2} we deduce that
\[
  \|u^\epsilon - u\|_{L^2(\R_+,r\dd r)}^2 \,=\, \frac{1}{2\pi} 
  \|U^\epsilon(\cdot,\cdot,0) - U(\cdot,\cdot,0)\|_{L^2(\R^2)}^2 
  \,\xrightarrow[\epsilon \to 0]{}\, 0\,,
\]
and this gives the desired result. 
\end{proof}

\subsection{Compactness criteria}\label{sec52}

We next mention two simple compactness criteria in the space 
$X = L^2(\R_+,r\dd r)$. 

\begin{lem}\label{lem:compcrit}
For any $\alpha > 0$ and any $M > 0$, the sets 
\begin{align*}
  E_{M,\alpha} \,&=\, \bigl\{f \in X\,;\, \|\partial_r f\|_{L^2} \le M\,,~ 
  \|r^\alpha f\|_{L^2} 
  \le M\bigr\}\,, \quad \hbox{and}\\ 
  E_{M,\alpha}^* \,&=\, \bigl\{f \in X\,;\, \|\partial_r^* f\|_{L^2} \le 
  M\,,~ \|r^\alpha f\|_{L^2} \le M\bigr\}\,,
\end{align*}
are compact in $X$. We recall that $\partial_r^* = \partial_r + \frac1r$. 
\end{lem}

\begin{proof}
If $f \in X$, we define $F : \R^2 \to \C$ by $F(x) = (2\pi)^{-1/2} f(|x|)$ 
for all $x \in \R^2$. The linear map $f \mapsto F$ is an isometric embedding 
of $X$ into $L^2(\R^2)$, and the image of $E_{M,\alpha}$ under that map is 
included in the set
\[
  \bigl\{F \in L^2(\R^2)\,;\, \|\nabla F\|_{L^2} \le M\,,~
  \| |x|^\alpha F\|_{L^2} \le M\bigr\}\,,
\]
which is known to be compact in $L^2(\R^2)$ by Rellich's criterion, 
see \cite[Theorem~XIII.65]{ReSi}. This shows that the closed subset
$E_{M,\alpha} \subset X$ is relatively compact, hence compact. 

Compactness of $E_{M,\alpha}^*$ can be established by a variant of 
the previous argument, but for a change we give here a direct proof
based on the Arzel\`a-Ascoli theorem. If $f \in E_{M,\alpha}^*$, we 
observe that
\begin{equation}\label{eq:fintrep}
  f(r) \,=\, \frac{1}{r}\int_0^r \partial_r^* f(s) s\dd s\,, \qquad 
  \hbox{for all } r > 0\,.
\end{equation}
This shows that $|f(r)| \le \|\partial_r^* f\|_{L^2} \le M$ for all 
$r > 0$, and we deduce that
\[
  \int_0^\epsilon |f(r)|^2 r\dd r \,\le\, M^2 \epsilon^2\,,
  \qquad 
  \int_L^\infty |f(r)|^2 r\dd r \,\le\, \frac{1}{L^{2\alpha}} \|r^\alpha f\|_{L^2}^2 
  \,\le\, \frac{M^2}{L^{2\alpha}}\,,
\]
for any $\epsilon > 0$ and any $L > 0$. In particular, the set
$E_{M,\alpha}^*$ is bounded in $X$, and its elements are uniformly
small near the origin and at infinity. Moreover, it follows from
\eqref{eq:fintrep} and H\"older's inequality that
\[
  |r_1 f(r_1) - r_2 f(r_2)| \,\le\, M |r_1 - r_2|^{1/2}\,,
  \qquad \hbox{for all } r_1, r_2  > 0\,,
\]
which means that the elements of $E_{M,\alpha}^*$ are {\em uniformly 
equicontinuous} on any compact interval $[\epsilon,L] \subset (0,\infty)$. 
These properties altogether imply that $E_{M,\alpha}^*$ is a compact 
subset of $X$. 
\end{proof}

\subsection{Representation formulas}\label{sec53}

Finally we give explicit representation formulas for the pressure
$p$ satisfying \eqref{eq:pmk}, in terms of solutions of the 
homogeneous equation
\begin{equation}\label{eq:homogene}
  -\partial_r^* \partial_r p(r) + \frac{m^2}{r^2}p(r) + k^2
  p(r) \,=\, 0\,.
\end{equation}
If $k \neq 0$, a pair of linearly independent solutions of 
\eqref{eq:homogene} is given by the modified Bessel functions $I_m(|k|r)$ 
and $K_m(|k|r)$, see e.g. \cite[Section~9.6]{AS}. For later use, 
we recall that $I_{-m}(r) = I_m(r)$, $K_{-m}(r) = K_m(r)$, and
$K_m(r) I_m'(r) -  K_m'(r) I_m(r) = 1/r$ for all $r > 0$. 
Moreover, if $m \ge 1$, then
\begin{equation}\label{eq:asym0}
  I_m(r) \,\sim\, \frac{1}{m!}\Bigl(\frac{r}{2}\Bigr)^m\,, \qquad
  K_m(r) \,\sim\, \frac{(m{-}1)!}{2}\Bigl(\frac{2}{r}\Bigr)^m\,, 
  \qquad \hbox{as } r \to 0\,,
\end{equation}
whereas $I_0(r) \to 1$ and $K_0(r) \sim -\log(r)$ as $r \to 0$. 
For all $m \in \Z$, we also have
\begin{equation}\label{eq:asyminf}
  I_m(r) \,\sim\, \frac{1}{\sqrt{2\pi}}\,\frac{e^r}{\sqrt{r}}\,, \qquad
  K_m(r) \,\sim\, \sqrt{\frac{\pi}{2}}\,\frac{e^{-r}}{\sqrt{r}}\,, 
  \qquad \hbox{as } r \to +\infty\,.
\end{equation}
When $k = 0$ linearly independent solutions of \eqref{eq:homogene} are 
$r^{\pm m}$ if $m \neq 0$, and $\{1,\log(r)\}$ if $m = 0$. 

\begin{lem}\label{lem:rep}
Assume that the vorticity profile $W$ satisfies assumption~H1. 
For any $m \in \Z$, $k \in \R$, and $u \in X_{m,k}$, the elliptic equation 
\eqref{eq:pmk} has a unique solution $p = P_{m,k}[u]$ such that 
$p(r) = \cO(|\log r|^{1/2})$ as $r \to 0$ and $p(r) \to 0$ as 
$r \to +\infty$. If $k \neq 0$, we have $p = 2im p_1 + 2|k|p_2$ where
\begin{equation}\label{eq:repmk}
\begin{split}
  p_1(r) \,&=\, K_m(|k|r) \int_0^r I_m(|k|s) (s\Omega)'
  u_r(s) \dd s + I_m(|k|r) \int_r^\infty K_m(|k|s) (s\Omega)' 
  u_r(s) \dd s\,, \\
  p_2(r) \,&=\, K_m(|k|r) \int_0^r I_m'(|k|s) \Omega(s) 
  u_\theta(s) s\dd s + I_m(|k|r) \int_r^\infty K_m'(|k|s) \Omega(s) 
  u_\theta(s) s\dd s\,.
\end{split}
\end{equation}
If $k = 0$ and $m \neq 0$, then $p = \sigma p_1 + p_2$ where
$\sigma = m/|m|$ and 
\begin{equation}\label{eq:repm0}
\begin{split}
  p_1(r) \,&=\, \frac{i}{r^{|m|}} \int_0^r s^{|m|} (s\Omega)'(s) u_r(s)\dd s +
  i r^{|m|} \int_r^\infty \frac{1}{s^{|m|}} (s\Omega)'(s) u_r(s)\dd s\,, \\
  p_2(r) \,&=\, \frac{1}{r^{|m|}} \int_0^r s^{|m|} \Omega(s) u_\theta(s)\dd s -
  r^{|m|} \int_r^\infty \frac{1}{s^{|m|}} \Omega(s) u_\theta(s)\dd s\,.
\end{split}
\end{equation}
Finally, if $k = m = 0$, then $p(r) = - 2 \int_r^\infty \Omega(s) 
u_\theta(s)\dd s$. 
\end{lem}

\begin{proof}
In view of \eqref{eq:pmk} we can suppose without loss of generality
that $k \ge 0$. If $k > 0$, we first assume that $u \in X_{m,k} \cap 
C^1_c(\Rp)$ and we consider the linear elliptic equation
\begin{equation}\label{eq:linell}
  -\partial_r^* \partial_r p(r) +\frac{m^2}{r^2}\,p(r) + k^2 p(r)
  \,=\, f(r)\,, \qquad r > 0\,,
\end{equation}
where $f = 2im(\partial_r^* \Omega)u_r - 2 \partial_r^* (\Omega\,u_\theta)$. 
The unique solution of \eqref{eq:linell} that is regular at the origin 
and decays to zero at infinity is
\begin{equation}\label{eq:frep}
  p(r) \,=\, K_m(kr) \int_0^r I_m(ks) f(s) s\dd s + I_m(kr) 
  \int_r^\infty K_m(ks) f(s) s\dd s\,, \qquad r > 0\,.
\end{equation}
Replacing $f$ by its expression and integrating by parts, we easily 
obtain the representation \eqref{eq:repmk}. The general case where 
$u$ is an arbitrary function in $X_{m,k}$ follows by a density argument, 
using Proposition~\ref{prop:approximation}. 

If $k = 0$ and $m \neq 0$, the solutions of the homogeneous equation
\eqref{eq:homogene} are $r^{|m|}$ and $r^{-|m|}$, instead of
$I_m(|k|r)$) and $K_m(|k|r)$. Proceeding exactly as above, we thus
arrive at \eqref{eq:repm0} instead of \eqref{eq:repmk}. Finally, if
$k = m = 0$, any solution of \eqref{eq:pmk} such that $\partial_r p 
\in L^2(\Rp,r\dd r)$ satisfies $\partial_r p = 2 \Omega u_\theta$, 
hence $p(r) = - 2 \int_r^\infty \Omega(s) u_\theta(s)\dd s$. In all 
cases, the solution of \eqref{eq:pmk} given by the above formulas
satisfies $p(r) = \cO(|\log r|^{1/2})$ as $r \to 0$ and $p(r) \to 0$ 
as $r \to +\infty$, and is unique in that class. 
\end{proof}


\end{document}